\documentclass[12pt,reqno]{amsart}
\usepackage{amsaddr}
\usepackage{amssymb,amsfonts,url}
\usepackage{amsmath}
\usepackage[all,arc]{xy}
\usepackage{enumerate}
\usepackage{mathrsfs}
\usepackage[utf8]{inputenc}
\usepackage{pdfsync}
\usepackage[margin=1in]{geometry}
\usepackage{caption}
\usepackage{color,graphicx}
\usepackage{hyperref}
\usepackage[T1]{fontenc}
\usepackage[sc,osf]{mathpazo}
\usepackage[section]{placeins}
\usepackage{mathtools}

\DeclarePairedDelimiter\floor{\lfloor}{\rfloor}

\newtheorem{thm}{Theorem}[section]

\newtheorem{lem}[thm]{Lemma}

\theoremstyle{definition}

\theoremstyle{remark}
\newtheorem{rem}[thm]{Remark}

\makeatletter
\let\c@equation\c@thm
\makeatother
\numberwithin{equation}{section}

\title[Biharmonic Schrödinger equation on the half-line]{The initial-boundary value problem for the biharmonic Schrödinger equation on the half-line }

\author{T. Özsarı\textsuperscript{$\star$} \& N. Yolcu\textsuperscript{$\dagger$}}
\address{Department of Mathematics, Izmir Institute of Technology, Urla, Izmir, TURKEY}
\thanks{\textsuperscript{$\star$}Correspondence: Türker Özsarı, Department of Mathematics, Izmir Institute of Technology, Urla, Izmir 35430, TURKEY; E-mail: turkerozsari@iyte.edu.tr}
\thanks{\textsuperscript{$\dagger$}The results of this paper will also appear as part of the Ph.D. thesis of Nermin Yolcu at Izmir Institute of Technology.}
\thanks{This research was funded by TÜBİTAK 1001 Grant \#117F449}

\date{}

\begin{document}

\begin{abstract}
We study the local and global wellposedness of the initial-boundary value problem for the biharmonic Schrödinger equation on the half-line with inhomogeneous Dirichlet-Neumann boundary data. First, we obtain a representation formula for the solution of the linear nonhomogenenous problem by using the Fokas method (also known as the \emph{unified transform method}). We use this representation formula to prove space and time estimates on the solutions of the linear model in fractional Sobolev spaces by using Fourier analysis. Secondly, we consider the nonlinear model with a power type nonlinearity and prove the local wellposedness by means of a classical contraction argument. We obtain Strichartz estimates to treat the low regularity case by using the oscillatory integral theory directly on the representation formula provided by the Fokas method.  Global wellposedness of the defocusing model is established up to cubic nonlinearities by using the multiplier technique and proving hidden trace regularities.
\end{abstract}

\keywords{fourth-order Schrödinger equation; biharmonic Schrödinger equation; Fokas method; unified transform method; local wellposedness; global wellposedness; space estimates; time estimates; Strichartz estimates; \and inhomogeneous boundary data}
\subjclass[2010]{35Q55, 35C15, 35A07, 35A22, 35G15, 35G30}
\maketitle
\tableofcontents

\section{Introduction}
This article studies the local and global wellposedness of the initial\,-\,(inhomogeneous) boundary value problem for the biharmonic nonlinear Schrödinger equation (NLS) which is posed on the right half-line:
\begin{align}
  &iq_{t}+\partial_x^{4}q=f(q), \quad (x,t)\in\mathbb{R}_+\times (0,T),\label{4th.1}\\
  &q(x,0)=q_{0}(x), \quad x\in\mathbb{R}_+,\label{4th.2}\\
  &q(0,t)=g_{0}(t), \quad  t\in(0,T),\label{4th.3} \\
  &q_{x}(0,t)=g_{1}(t), \quad  t\in(0,T),\label{4th.4}
\end{align} where $f(q)=\kappa|q|^pq$, $\kappa\in\mathbb{C}$, $T,p>0$, and $q$ is a complex valued function.  The analysis here is carried out in the $L^2-$based fractional Sobolev space $H^s(\mathbb{R}_+)$ at the spatial level, where throughout the paper (without any restatement) we will assume the following in order to work with a sufficiently nice nonlinearity:
\begin{itemize}
  \item[(a1)] if $s$ is integer, then $p\ge s$ if $p$ is an odd integer and $\floor{p}\ge s-1$ if $p$ is non-integer,
  \item[(a2)] if $s$ is non-integer, then $p>s$ if $p$ is an odd integer and $\floor{p}\ge \floor{s}$ if $p$ is non-integer.
\end{itemize}

The fourth-order NLS, in the form
\begin{equation}\label{mixeddisp}iu_{t}+\Delta u+\gamma \Delta^2u+|u|^pu=0,\,x\in \mathbb{R}^n,\,t\in\mathbb{R}\end{equation}
was introduced by \cite{Karpman96}-\cite{Karpman2000} to study the stabilizing role of the higher-order dispersive effects. It was shown that the solutions are stable if $\gamma<0$, $\displaystyle p\le\frac{4}{n}$ or if $\gamma\ll-1$, $\displaystyle p\in \left(\frac{4}{n},\frac{8}{n}\right)$. Moreover, solutions were found to be unstable for $\displaystyle\gamma<0$, $\displaystyle p\ge\frac{8}{n}$ in which case solutions may cease to exist globally.

In the absence of the Laplacian, the fourth order NLS takes the form \begin{equation}\label{biharmonic}iu_{t}+\gamma \Delta^2u+|u|^pu=0,\,x\in \mathbb{R}^n,\,t\in\mathbb{R}\end{equation} and it is called the biharmonic NLS. It was shown by \cite{Fin2002} (see also \cite{Fin2011} and the references therein) that all solutions of the biharmonic NLS are global if $\gamma>0$.  Moreover, it was found that $\displaystyle p=\frac{8}{n}$ is the critical exponent for singularity formation if $\gamma<0$, and smallness in the mean-square sense is sufficient for global existence if $\displaystyle p=\frac{8}{n}$.  The biharmonic NLS $$iu_{t}+\gamma \Delta^2u+\kappa|u|^pu=0,$$ with $\gamma,\kappa\in\mathbb{R}$, is said to be focusing (resp. defocusing) if $\gamma\kappa<0$ (resp. $\gamma\kappa>0$).

The rigourous analysis of the solutions of the fourth order Schrödinger equation started with the proof of sharp space-time decay properties for the linear group associated to the operator $i\partial_t+\lambda \Delta+\Delta^2$, where $\lambda\in \{-1,0,1\}$ \cite{Ben2000}.  One can actually use these properties to obtain Strichartz estimates, which gives the local wellposedness at $H^2$-level.

Local well-posedness of the nonlinear fourth order Schrödinger equations in one space dimension was studied in \cite{Seg04}, \cite{Hao2006}, and \cite{Zheng11}.  Global well-posedness in one dimensional case with small initial data was proved with various nonlinearities in \cite{Hayashi15}, \cite{Hayashi15-2}, \cite{Hayashi15-3}, \cite{Hayashi15-4}, and \cite{aoki16}.  Local well-posedness in the muti-dimensional case was treated in \cite{Hao2007} and the global well-posedness was studied in \cite{cui2007}, \cite{Pausader10}, \cite{guo2010}, \cite{guo12}, and \cite{Zhang2010}.  Global well-posedness at the $H^2-$level in the energy-critical case with power-type nonlinearities was shown by \cite{Pausader07} for radial initial data.  Global well-posedness and ill-posedness of the cubic defocusing biharmonic NLS was studied in \cite{Pausader09}.  It turns out that the cubic defocusing problem is ill-posed in dimensions $n\ge 9$, and well-posed in dimensions $n\le 8$, while the scattering holds true for dimensions $5\le n\le 8$. Other scattering results in the one dimensional scenario were obtained by \cite{Seg06} and \cite{Seg06-2}, while the high dimensional scattering problems  were studied in \cite{Pausader09}, \cite{Miao2009}, \cite{Miao2015}, \cite{Wang2012}, \cite{Ruz16}, \cite{Pausader13}, and \cite{Pausader07-2}. The last paper in particular proves the Levandosky-Strauss conjecture in the defocusing case. The blow-up phenomenon for the biharmonic NLS was studied in \cite{Fin2011}, \cite{Zhul2010}, \cite{Zhu11}, \cite{Zhu11-2}, \cite{dinh17}, and \cite{boul17}.

The references given above studied the fourth order Schrödinger equation in the whole Euclidean space, namely the spatial domain was assumed to be equal to $\mathbb{R}^n$. The absence of the boundary in these studies simplified the mathematical and physical analysis of the problem to some extent.  However, in order to boost the physical reality, it is common to assume that the evolution takes place in a region with boundary, and what happens at the boundary influences the nature of the solutions.  This is especially important for a control scientist, since boundary can be used as a control point, particularly when it is difficult or impossible to access the medium of the evolution.  This idea motivated some of the recent studies related with the controllability of the the solutions of the linear fourth order Schrödinger equation.  For instance, \cite{Wen16-2}, \cite{Wen14}, and \cite{Wen16} studied  the well-posedness and  exact controllability of the linear biharmonic Schrödinger equation on a bounded domain $\Omega\subset \mathbb{R}^n$.  Most recently, \cite{Aksas17} studied the stabilization of the linear biharmonic Schrödinger equation on a bounded domain with a locally supported internal damping.

From the physical point of view, the model under consideration in this paper corresponds to a situation in which the wave is generated from a fixed source such that the wave train moves into the medium in one specific direction. Wellposedness of similar inhomogenenous initial boundary value problems on the half-line were recently considered for the classical Schrödinger equation; see for example \cite{Carroll}, \cite{bu92}, \cite{bona}, \cite{holmer}, and \cite{fokas}. We prove the corresponding wellposedness theorems for the biharmonic Schrödinger equations, and as far as we know this is the first treatment of the fourth order Schrödinger equations subject to inhomogeneous boundary conditions.

\subsection{Main results}

In this paper, attention is given only to the biharmonic NLS. More general fourth order Schrödinger equations with mixed dispersion as in \eqref{mixeddisp} will be taken into consideration in a further study.  Our first main result is the local well-posedness of solutions in fractional Sobolev spaces. More precisely, we prove the following theorem.

\begin{thm}[Local wellposedness I] Let $T>0$, $s\in \left(\frac{1}{2},\frac{9}{2}\right)$, $s\neq \frac{3}{2}$, $p>0$, $q_0\in H^s(\mathbb{R_+})$, $\displaystyle g_0\in H^{\frac{2s+3}{8}}(0,T)$, $\displaystyle g_1\in H^{\frac{2s+1}{8}}(0,T)$, $q_0(0)=g_0(0)$, (also $q_0'(0)=g_1(0)$ if $s>\frac{3}{2}$).  Then, \eqref{4th.1}-\eqref{4th.4} has the following local wellposedness properties:
\begin{itemize}
  \item[(i)] Local existence and uniqueness:  there exists a unique local solution $q\in C([0,T_0];H^s(\mathbb{R_+}))$ for some ${T_0\in (0,T]}$,
  \item[(ii)] Continuous dependence:  if $B$ is a bounded subset of $H^s(\mathbb{R}_+)\times H^{\frac{2s+3}{8}}(0,T_0)\times H^{\frac{2s+1}{8}}(0,T_0)$, then there is $T_0>0$ such that the flow $(q_0,g_0,g_1)\rightarrow q$ is Lipschitz continuous from $B$ into $C([0,T_0];H^s(\mathbb{R_+}))$,
  \item[(iii)] Blow-up alternative: Let $S$ be the set of all $T_0\in (0,T]$ such that there exists a unique local solution in  $C([0,T_0];H^s(\mathbb{R_+}))$.  If $\displaystyle T_{max}:=\sup_{T_0\in S}T_0<T$, then $\displaystyle\lim_{t\uparrow T_{max}}\|q(t)\|_{H^s(\mathbb{R}_+)}=\infty.$
\end{itemize}
 \end{thm}
 \begin{rem}
The proof of the above theorem is based on the Fokas method  (\cite{fokas1}, \cite{fokasb}) combined with classical contraction arguments. The Fokas method is a unified approach for solving initial-(inhomogeneous) boundary value problems for a general class of linear evolution equations. It has significant advantages over traditional methods.  One of these advantages is that the solution formula obtained via the Fokas method is uniformly convergent at the boundary points. This is an important property for numerical studies.

Another remarkable feature of the Fokas method is that one can obtain the necessary space and time estimates for the corresponding linear evolution operator directly from the representation formula by using Fourier analysis.  This allows one to easily study the wellposedness of the initial-boundary value problem for corresponding nonlinear models.  The study of rigorous wellposedness analysis of nonlinear initial-boundary value problems using the Fokas method was initiated by \cite{fokasKdV}, \cite{fokas}.  In these studies, the authors determine the fractional Sobolev spaces for initial and boundary data for which the local Hadamard well-posedness holds true.  Similar results have been obtained for other models such as the \emph{good} Boussinesq equation \cite{himonas15}, and two dimensional nonlinear Schrödinger \cite{himarx1} and reaction diffusion equations \cite{himarx2}.
\end{rem}

Second, we prove the local existence and uniqueness at the low regularity setting $s<1/2$.
\begin{thm}[Low regularity]\label{thmlowreg} Let $T>0$, $s\in [0,\frac{1}{2})$, $p\in (0,\frac{8}{1-2s}]$, $q_0\in H^s(\mathbb{R_+})$, $\displaystyle g_0\in H^{\frac{2s+3}{8}}(0,T)$, $\displaystyle g_1\in H^{\frac{2s+1}{8}}(0,T)$, $\lambda=\frac{8(p+2)}{p(1-2s)}$, and $r=\frac{p+2}{1+sp}$. Then, \eqref{4th.1}-\eqref{4th.4} has a unique solution $q\in C([0,T_0];H^s(\mathbb{R_+}))\cap L^\lambda(0,T_0;W^{s,r}(\mathbb{R}_+))$ for some ${T_0\in (0,T]}$.
 \end{thm}

Finally, we prove the global wellposedness of weak (more precisely $H^2$) solutions for the defocusing problem.  We are able to prove the global wellposedness for $p\le 2$, as opposed to $p>0$, which is the case for the problem posed with homogeneous boundary conditions.
\begin{thm}[Global wellposedness] Let $\kappa\in \mathbb{R}_{-}$ (defocusing nonlinearity), $T>0$, $p\le 2$, $q_0\in H^2(\mathbb{R_+})$, $g_0\in H^1(0,T)$, $g_1\in H^1(0,T)$, $q_0(0)=g_0(0)$, $q_0'(0)=g_1(0)$. Then, the corresponding local solution $q\in C([0,T];H^2(\mathbb{R}_+))$ is global. Moreover, the global solution $q$ satisfies the hidden trace regularities given by $q_{xx}(0,\cdot), q_{xxx}(0,\cdot)\in L^2(0,T)$.
 \end{thm}
\begin{rem}
  The global wellposedness problem for the nonlinear biharmonic Schrödinger equation is a nontrivial problem in the presence of inhomogeneous boundary conditions as opposed to the case of homogeneous boundary conditions.  The main difficulty lies in the fact that one looses all energy conservation and control properties once $g_0$ and $g_1$ are non-zero.  For instance, the most basic energy identiy, namely the $L^2$-energy, satisfies an equality given by  $$\frac{1}{2}\frac{d}{dt}\int_0^\infty|q|^2dx=\text{Im}\left[q_{xxx}(0,t)\bar{g}_0(t)\right]-\text{Im}\left[q_{xx}(0,t)\bar{g}_1(t)\right].$$  This identity involves the unknown traces $q_{xxx}(0,t)$ and $q_{xx}(0,t)$. Higher order energy estimates are even more complicated.

  The extra regularity result  $q_{xx}(0,\cdot), q_{xxx}(0,\cdot)\in L^2(0,T)$ proved in the above theorem is called a hidden trace regularity since in general for an arbitrary $H^2$ function, these traces do not need to be well-defined in the sense of Sobolev trace theory. Hence, this shows that the biharmonic Schrödinger operator has a hidden regularizing property in the sense of traces.
\end{rem}
\subsection{Orientation}
We prove the main results in several steps:
\subsubsection*{Step 1 - Representation formula.} We use the Fokas method (also known as the \emph{unified transform method}) to obtain a representation formula for the solution of the linear biharmonic Schrödinger equation with interior force and inhomogenenous boundary inputs. The derivation of the representation formula is more complicated than the classical Schrödinger equation due to the higher order nature of the evolution operator.
\subsubsection*{Step 2 - Cauchy problem.} Secondly, we study the Cauchy problem on the spatial domain $\mathbb{R}$. We obtain the necessary space and times estimates on the solutions of the Cauchy problem.  These estimates are later used to study the half-line problem with zero boundary inputs by extending the given initial datum from half-line to the whole line.
\subsubsection*{Step 3 - Half line problem.} We use the representation formula obtained in Step 1 to study the half-line problem with zero initial data and inhomogeneous Dirichlet-Neumann boundary inputs. We obtain space and time estimates on the solutions of the half-line problem.  The analysis poses more challenges than the Schrödinger equation with only Dirichlet input, because the inhomogeneous Neumann input here requires us to deal with integrals that involve singular integrands.  Singularities are treated with cut-off functions.
\subsubsection*{Step 4 - Operator theoretic formula.} In this step, we express the representation formula of the linear problem with internal force in operator theoretic form.  Then, we replace the internal source with the given nonlinearity and applying the contraction argument to this form of the representation formula to obtain the local wellposedness.
\subsubsection*{Step 5 - Strichartz estimates.} We treat the low regularity case $s<\frac{1}{2}$ for the nonlinear model by proving Strichartz estimates by using the oscillatory integral theory.
\subsubsection*{Step 6 - Global wellposedness.} Finally we use the multiplier method to obtain useful energy estimates that involve information on the unknown boundary traces.  The classical multipliers for the biharmonic Schrödinger equations are not sufficient.  Therefore, we use also the control theoretic multipliers to prove hidden trace regularities for the second and third order boundary traces of the solution.  We combine these with $H^2$ energy identities to deduce the global wellposedness of solutions up to cubic powers.  The results obtained here are quite interesting compared to solutions of the homogeneous boundary value problem, whose solutions can be shown to be global for all powers in the defocusing case.

\section{Linear model}
In this section, we consider the linear biharmonic Schrödinger equation on the right half-line with inhomogeneous Dirichlet-Neumann data in $L^2-$based fractional Sobolev spaces:
 \begin{align}
  &iq_{t}+\partial_x^4q=f, \quad (x,t)\in\mathbb{R}_+\times (0,T),\label{1.1}\\
  &q(x,0)=q_{0}(x), \quad x\in\mathbb{R}_+,\label{1.2}\\
  &q(0,t)=g_{0}(t), \quad  t\in(0,T),\label{1.3} \\
  &q_{x}(0,t)=g_{1}(t), \quad  t\in(0,T).\label{1.4}
  \end{align}

\subsection{Representation formula} We will obtain a representation formula for the solution of \eqref{1.1}-\eqref{1.4}.  In order to do this, we will use the Fokas (unified transform) method.  To this end, let $\hat{q}(k,t)$ denote the spatial Fourier transform of the solution of \eqref{1.1}-\eqref{1.4} on the right half-line:
\begin{equation}\label{qhatkt}
  \hat{q}(k,t) \equiv \int_0^\infty e^{-ikx}q(x,t)dx, \quad \text{Im}\,k\le 0.
\end{equation}  Note that the condition $\text{Im}\,k\le 0$ is enforced so that the right hand side of the formula \eqref{qhatkt} will in general converge.  Applying this transform to the problem \eqref{1.1}-\eqref{1.4}, after some computations we get

\begin{multline}\label{qthatkt}
  i\hat{q}_t(k,t)+k^4\hat{q}(k,t)-q_{xxx}(0,t)-ikq_{xx}(0,t)+k^2q_x(0,t)+ik^3q(0,t)=\hat{f}(k,t), \quad \text{Im}\,k\le 0.
\end{multline}

Integrating the above equation in the temporal variable, we obtain

\begin{multline}\label{Intqthatkt}
  e^{-ik^4t}\hat{q}(k,t)=\hat{q}_0(k)\\
  -\int_0^te^{-ik^4s}\left[iq_{xxx}(0,s)-kq_{xx}(0,s)-ik^2q_x(0,s)+k^3q(0,s)+i\hat{f}(k,s)\right]ds, \quad \text{Im}\,k\le 0.
\end{multline} Introducing the formula \begin{equation}\label{qtildej}
                                          \tilde{g_{j}}(k,t)=\int_{0}^{t}e^{ks}\partial_x^jq(0,s)ds,\quad k\in\mathbb{C}, \,j= \overline{0,3},
                                        \end{equation} we can rewrite \eqref{Intqthatkt} as
\begin{multline}\label{IntqthatktRe}
e^{-ik^4t}\hat{q}(k,t)=\hat{q}_0(k)-i\tilde{g}_3(-ik^4,t)+k\tilde{g}_2(-ik^4,t)+ik^2\tilde{g}_1(-ik^4,t)-k^3\tilde{g}_0(-ik^4,t)\\
-i\int_0^te^{-ik^4s}\hat{f}(k,s)ds, \quad \text{Im}\,k\le 0.
\end{multline}

Multiplying both sides of \eqref{IntqthatktRe} by $e^{ik^4t}$ and then taking the inverse Fourier transform, we get
\begin{multline}\label{qxt}
{q}(x,t)=\frac{1}{2\pi}\int_{-\infty}^\infty e^{ikx+ik^4t}\hat{q}_0(k)dk\\
-\frac{1}{2\pi}\int_{-\infty}^\infty e^{ikx+ik^4t}\left[i\tilde{g}_3(-ik^4,t)-k\tilde{g}_2(-ik^4,t)-ik^2\tilde{g}_1(-ik^4,t)+k^3\tilde{g}_0(-ik^4,t)\right]dk\\
-\frac{i}{2\pi}\int_{-\infty}^\infty e^{ikx+ik^4t}\left[\int_0^te^{-ik^4s}\hat{f}(k,s)ds\right]dk, 0<x<\infty, t>0.
\end{multline}

The above formula consists of two unknown boundary traces, namely $q_{xxx}(0,t)$ and $q_{xx}(0,t)$.  The idea of the uniform transform method is based on eliminating these unknown quantities from the solution formula.  This is achieved in two steps:

\begin{enumerate}
  \item[(1)] We deform the contour of integration involving unknown quantities from $(-\infty,\infty)$ to another appropriate contour.
  \item[(2)] We take advantage of the invariant properties of equation \eqref{IntqthatktRe} satisfied by $\hat{q}(k,t)$.
\end{enumerate}

To this end, we first consider the region $D$ described by
\begin{equation} D\equiv \{k\in\mathbb{C}\,|\,\text{Re}\,(-ik^4)<0\}.
\end{equation}  It is easy to show that the above region can also be written by

\begin{equation} D\equiv \left\{k\in\mathbb{C}\,|\,\text{Arg}\,k\in \bigcup_{m=0}^3\left(\frac{(2m+1)\pi}{4},\frac{(m+1)\pi}{2}\right)\right\},
\end{equation} where the principle argument of a complex number is assumed to be defined in the interval $[0,2\pi).$

Now, we split $D$ in two disjoint parts depending on whether $k\in D$ is in the upper or lower half-plane: $D^+\equiv D\cap \mathbb{C}^+$ and $D^-\equiv D\cap \mathbb{C}^-$.  See Figure \ref{D+D-} below.

\begin{figure}[h]
  \centering
   \includegraphics[scale=.75]{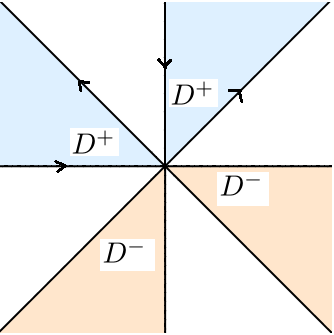}
  \caption{The region $D=D^+\cup D^-$}\label{D+D-}
\end{figure}

The following lemma follows from the the unified theory given in \cite[Proposition 1.1, Chapter 1]{fokasb}:

\begin{lem}[Deformation]
  Let $q$ be a solution of \eqref{1.1} on $\Omega\equiv \mathbb{R}_+\times (0,T)$ such that $q$ is sufficiently smooth up to the boundary of $\Omega$ and decays sufficiently fast as $x\rightarrow \infty$, uniformly in $[0,T]$.  Then, $q(x,t)$ can be represented by
\begin{multline}\label{qDeformed}
{q}(x,t)=\frac{1}{2\pi}\int_{-\infty}^\infty e^{ikx+ik^4t}\hat{q}_0(k)dk
-\frac{i}{2\pi}\int_{-\infty}^\infty e^{ikx+ik^4t}\left[\int_0^te^{-ik^4s}\hat{f}(k,s)ds\right]dk\\
-\frac{1}{2\pi}\int_{\partial D^+} e^{ikx+ik^4t}\tilde{g}(k)dk,
\end{multline} where $\tilde{g}(k)=i\tilde{g}_3(-ik^4,T)-k\tilde{g}_2(-ik^4,T)-ik^2\tilde{g}_1(-ik^4,T)+k^3\tilde{g}_0(-ik^4,T)$ and $\partial D^+$ is oriented in such a way that $D^+$ is to the left of $\partial D^+$.
\end{lem}

Now, we will use the invariant properties of the equation \eqref{IntqthatktRe}.  Replacing $k$ by $-k$ in this equation keeps $\tilde{g}_j(-ik^4,T)$ invariant for $j = \overline{0,3}$, and one obtains
\begin{multline}\label{Intqthatkt-b}
  e^{-ik^4T}\hat{q}(-k,T)=\hat{q}_0(-k)-i\tilde{g}_3(-ik^4,T)-k\tilde{g}_2(-ik^4,T)+ik^2\tilde{g}_1(-ik^4,T)+k^3\tilde{g}_0(-ik^4,T)\\
  -i\int_0^Te^{-ik^4s}\hat{f}(-k,s)ds, \quad \text{Im}\,k\ge 0.
\end{multline} Similarly, replacing $k$ by $ik$ and $-ik$, one can keep $\tilde{g}_j(-ik^4,T)$ invariant for $j = \overline{0,3}$.  Moreover, we have the identities
\begin{multline}\label{Intqthatkt-c}
  e^{-ik^4T}\hat{q}(ik,T)=\hat{q}_0(ik)-i\tilde{g}_3(-ik^4,T)+ik\tilde{g}_2(-ik^4,T)-ik^2\tilde{g}_1(-ik^4,T)+ik^3\tilde{g}_0(-ik^4,T)\\
  -i\int_0^Te^{-ik^4s}\hat{f}(ik,s)ds, \quad \text{Re}\,k\le 0,
\end{multline} and
\begin{multline}\label{Intqthatkt-d}
  e^{-ik^4T}\hat{q}(-ik,T)=\hat{q}_0(-ik)-i\tilde{g}_3(-ik^4,T)-ik\tilde{g}_2(-ik^4,T)-ik^2\tilde{g}_1(-ik^4,T)-ik^3\tilde{g}_0(-ik^4,T)\\
  -i\int_0^Te^{-ik^4s}\hat{f}(-ik,s)ds, \quad \text{Re}\,k\ge 0,
\end{multline}respectively. Let us define two subregions of $D^+$ (See Figure \ref{D+D-}) by $$D_1^+:=D^+\cap \{\text{Re}\,k\ge 0\}\text{ and }D_2^+:=D^+\cap \{\text{Re}\,k\le 0\}$$ and let us also set the following transformation for the given boundary data: \begin{equation}\label{Gj}G_j(k,t):=\int_0^te^{ks}g_j(s)ds,\,k\in\mathbb{C},j=0,1.\end{equation}

\begin{figure}[h]
  \centering
   \includegraphics[scale=.75]{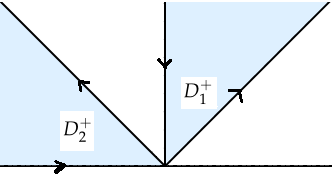}
  \caption{The region $D_1^+$ and $D_2^+$}\label{D1+D2+}
\end{figure}

The identities \eqref{Intqthatkt-b} and \eqref{Intqthatkt-d} are valid in $D_1^+$. Solving these identities for $\tilde{g}_2$ and $\tilde{g}_3$ and using the fact that $G_j=\tilde{g}_j$ for $j=0,1$, we obtain
\begin{multline}\label{g2tilded1}\tilde{g}_2(-ik^4,T)=\frac{e^{-ik^4T}}{k(1-i)}\left(\hat{q}(-ik,T)-\hat{q}(-k,T)\right)+\frac{\hat{q}_{0}(-k)-\hat{q}_{0}(-ik)}{k(1-i)}\\
+\frac{k^{2}(1+i)}{1-i}G_0(-ik^{4},T)+\frac{2ik}{1-i}G_1(-ik^{4},T)+\frac{i-1}{2k}\int_0^Te^{-ik^4s}\left[\hat{f}(-ik,s)-\hat{f}(-k,s)\right]ds\end{multline} and
\begin{multline}\label{g3tilded1}\tilde{g}_3(-ik^4,T)=-\frac{e^{-ik^4T}}{1+i}\left(\hat{q}(-ik,T)-i\hat{q}(-k,T)\right)+\frac{-i\hat{q}_{0}(-k)+\hat{q}_{0}(-ik)}{1+i}\\
-\frac{2ik^{3}}{1+i}G_0(-ik^{4},T)+\frac{k^2(1-i)}{1+i}G_1(-ik^{4},T)+\frac{i-1}{2}\int_0^Te^{-ik^4s}\left[i\hat{f}(-ik,s)+\hat{f}(-k,s)\right]ds\end{multline} for $k\in D_1^+$. Similarly, by using the identities \eqref{Intqthatkt-b} and \eqref{Intqthatkt-c}, which are valid on $D_2^+$, we have
\begin{multline}\label{g2tilded2}\tilde{g}_2(-ik^4,T)=\frac{e^{-ik^4T}}{k(1+i)}\left(\hat{q}(ik,T)-\hat{q}(-k,T)\right)+\frac{\hat{q}_{0}(-k)-\hat{q}_{0}(ik)}{k(1+i)}\\
+\frac{k^{2}(1-i)}{1+i}G_0(-ik^{4},T)+\frac{2ik}{1+i}G_1(-ik^{4},T)+\frac{1+i}{2k}\int_0^Te^{-ik^4s}\left[\hat{f}(ik,s)-\hat{f}(-k,s)\right]ds\end{multline} and
\begin{multline}\label{g3tilded2}\tilde{g}_3(-ik^4,T)=\frac{e^{-ik^4T}}{1-i}\left(\hat{q}(ik,T)+i\hat{q}(-k,T)\right)-\frac{i\hat{q}_{0}(-k)+\hat{q}_{0}(ik)}{1-i}\\
-\frac{2ik^{3}}{1-i}G_0(-ik^{4},T)+\frac{k^2(1+i)}{1-i}G_1(-ik^{4},T)+\frac{1+i}{2}\int_0^Te^{-ik^4s}\left[i\hat{f}(ik,s)-\hat{f}(-k,s)\right]ds\end{multline} for $k\in D_2^+$. Using \eqref{g2tilded1}-\eqref{g3tilded2} in \eqref{qDeformed} together with the fact that  $G_j=\tilde{g}_j$ for $j=0,1$, we deduce the following identity:
\begin{multline}\label{qDeformednew}
{q}(x,t)=\frac{1}{2\pi}\int_{-\infty}^\infty e^{ikx+ik^4t}\hat{q}_0(k)dk-\frac{1}{2\pi}\int_{\partial D_1^+} e^{ikx+ik^4(t-T)}\left[-(1+i)\hat{q}(-ik,T)+i\hat{q}(-k,T)\right]dk\\
-\frac{i}{2\pi}\int_{-\infty}^\infty e^{ikx+ik^4t}\left[\int_0^te^{-ik^4s}\hat{f}(k,s)ds\right]dk-\frac{1}{2\pi}\int_{\partial D_1^+} e^{ikx+ik^4t}\left[(1+i)\hat{q}_0(-ik)-i\hat{q}_0(-k)\right]dk\\
-\frac{1}{2\pi}\int_{\partial D_1^+} e^{ikx+ik^4t}\left[2k^3(1-i)G_0(-ik^4,T)+2k^2\left(1-i\right)G_1(-ik^4,T)\right]dk\\
-\frac{1}{2\pi}\int_{\partial D_1^+}e^{ikx+ik^4t}\left[\int_0^Te^{-ik^4s}\left((1-i)\hat{f}(-ik,s)-\hat{f}(-k,s)\right)ds\right]dk\\
-\frac{1}{2\pi}\int_{\partial D_2^+} e^{ikx+ik^4(t-T)}\left[(i-1)\hat{q}(ik,T)-i\hat{q}(-k,T)\right]dk\\
-\frac{1}{2\pi}\int_{\partial D_2^+} e^{ikx+ik^4t}\left[(1-i)\hat{q}_0(ik)+i\hat{q}_0(-k)\right]dk\\
-\frac{1}{2\pi}\int_{\partial D_2^+} e^{ikx+ik^4t}\left[2k^3(1+i)G_0(-ik^4,T)-2k^2(1+i)G_1(-ik^4,T)\right]dk\\
-\frac{1}{2\pi}\int_{\partial D_2^+}e^{ikx+ik^4t}\left[\int_0^Te^{-ik^4s}\left(\hat{f}(-k,s)-(1+i)\hat{f}(ik,s)\right)ds\right]dk.
\end{multline}
The second and the seventh integrals at the right hand side of \eqref{qDeformednew} vanish by Cauchy's theorem and one obtains the following representation formula where the right hand side includes information coming only from the prescribed  initial-boundary-interior data:
\begin{multline}\label{qDeformednew2}
{q}(x,t)=\frac{1}{2\pi}\int_{-\infty}^\infty e^{ikx+ik^4t}\hat{q}_0(k)dk-\frac{i}{2\pi}\int_{-\infty}^\infty e^{ikx+ik^4t}\left[\int_0^te^{-ik^4s}\hat{f}(k,s)ds\right]dk\\-\frac{1}{2\pi}\int_{\partial D_1^+} e^{ikx+ik^4t}\left[(1+i)\hat{q}_0(-ik)-i\hat{q}_0(-k)\right]dk\\
-\frac{1}{\pi}\int_{\partial D_1^+} e^{ikx+ik^4t}\left[k^3(1-i)G_0(-ik^4,T)+k^2\left(1-i\right)G_1(-ik^4,T)\right]dk\\
-\frac{1}{2\pi}\int_{\partial D_1^+}e^{ikx+ik^4t}\left[\int_0^Te^{-ik^4s}\left((1-i)\hat{f}(-ik,s)-\hat{f}(-k,s)\right)ds\right]dk\\
-\frac{1}{2\pi}\int_{\partial D_2^+} e^{ikx+ik^4t}\left[(1-i)\hat{q}_0(ik)+i\hat{q}_0(-k)\right]dk\\
-\frac{1}{\pi}\int_{\partial D_2^+} e^{ikx+ik^4t}\left[k^3(1+i)G_0(-ik^4,T)-k^2(1+i)G_1(-ik^4,T)\right]dk\\
-\frac{1}{2\pi}\int_{\partial D_2^+}e^{ikx+ik^4t}\left[\int_0^Te^{-ik^4s}\left(\hat{f}(-k,s)-(1+i)\hat{f}(ik,s)\right)ds\right]dk.
\end{multline}

\subsection{Cauchy problem} In this section, we consider the biharmonic linear Schrödinger equation on the whole real line $\mathbb{R}$:
 \begin{align}
  &iy_{t}+\partial_x^4y=0, \quad (x,t)\in\mathbb{R}\times (0,T),\label{w1.1}\\
  &y(x,0)=y_{0}(x), \quad x\in\mathbb{R}.\label{w1.2}
  \end{align}
One has the following regularity properties regarding the Cauchy problem \eqref{w1.1}-\eqref{w1.2}.
\begin{lem}\label{Wrty0}
  Let $s\in \mathbb{R}$ and $y_0\in H^s(\mathbb{R})$.  Then the solution of \eqref{w1.1}-\eqref{w1.2}, denoted $y(t)=S_{\mathbb{R}}(t)y_0$, satisfies $y\in C([0,T];H^s(\mathbb{R}))$ with the conservation law \begin{equation}\label{conservationlaw}
                              \|y(\cdot,t)\|_{H^s(\mathbb{R})}=\|y_0\|_{H^s(\mathbb{R})},\,\,\,t\in [0,T].
                            \end{equation}
If $s\ge -\frac{3}{2}$, then $\displaystyle y\in C(\mathbb{R};H^{\frac{2s+3}{8}}(0,T))$ and there exists a constant $c_s\ge 0$ such that  \begin{equation}\label{timeest001}\sup_{x\in\mathbb{R}}\|y(x,\cdot)\|_{H^{\frac{2s+3}{8}}(0,T)}\le c_s(1+T^{\frac{1}{2}})\|y_{0}\|_{H^{s}(\mathbb{R})}.\end{equation} Moreover, if $s\ge -\frac{1}{2}$, then $y$ has the additional regularity $y_x\in C(\mathbb{R};H^{\frac{2s+1}{8}}(0,T))$ with the estimate
  \begin{equation}\label{timeest002}\sup_{x\in\mathbb{R}}\|\partial_xy(x,\cdot)\|_{H^{\frac{2s+1}{8}}(0,T)}\le c_s(1+T^{\frac{1}{2}})\|y_{0}\|_{H^{s}(\mathbb{R})}\end{equation} for some constant $c_s\ge 0$.
\end{lem}
\begin{proof} We prove this lemma by using arguments similar to the ones given in the proof of \cite[Theorem 4]{fokas}, which are based on the Fourier representation of the solution. There are a few crucial differences, particularly in the temporal regularity, compared to the classical Schrödinger equation due to the biharmonic evolution operator considered here.

Upon taking the Fourier transform of \eqref{w1.1}-\eqref{w1.2} in the spatial variable, we find
\begin{eqnarray}
\hat{y}(\xi,t)=e^{i\xi^{4}t}\hat{y}_{0}(\xi)\label{2.3}
\end{eqnarray} where $\hat{y}_{0}$ is the Fourier transform of $y_0$. Note that $$|\hat{y}(\xi,t)|=|e^{i\xi^{4}t}\hat{y}_{0}(\xi)|=|\hat{y}_0(\xi)|$$ for $\xi\in\mathbb{R}$, $t\in [0,T]$.  Therefore,
\begin{equation}\|y(t)\|_{H^s(\mathbb{R})}=\|y_0\|_{H^s(\mathbb{R})}\end{equation} and $y(t)\in H^s(\mathbb{R})$ for $t\in [0,T]$. Moreover, $t\rightarrow y(\cdot,t)$ is continuous from $[0,T]$ into $H^s(\mathbb{R}).$  In order to see this, let $t,t_n\in [0,T]$ be such that $t_n\rightarrow t$.  Then,
\begin{eqnarray*}
\|y(t_{n})-y(t)\|_{H^{s}(\mathbb{R})}^2=\int_{\mathbb{R}}(1+\xi^{2})^{s}|e^{i\xi^{4}t_{n}}-e^{i\xi^{4}t}|^{2}|\hat{y}_{0}(\xi)|^{2}d\xi,
\end{eqnarray*} where $$\displaystyle\lim_{n\rightarrow\infty}(1+\xi^{2})^{s}|e^{i\xi^{4}t_{n}}-e^{i\xi^{4}t}|^{2}|\hat{y}_{0}(\xi)|^{2}\rightarrow 0$$ for $\xi\in\mathbb{R}$ and  $$(1+\xi^{2})^{s}|e^{i\xi^{4}t_{n}}-e^{i\xi^{4}t}|^{2}|\hat{y}_{0}(\xi)|^{2}\le 4(1+\xi^{2})^{s}|\hat{y}_{0}(\xi)|^{2}$$ for $\xi\in\mathbb{R}.$ Note that the right hand side of the above inequality is a nonnegative integrable function since $$4\int_{\mathbb{R}}(1+\xi^{2})^{s}|\hat{y}_{0}(\xi)|^{2}d\xi=4\|y_0\|_{H^s(\mathbb{R})}^2<\infty.$$ Hence, by the dominated convergence theorem $\displaystyle\lim_{n\rightarrow \infty}\|y(t_{n})-y(t)\|_{H^{s}(\mathbb{R})}=0$. In other words, $y(t_n)\rightarrow y(t)$ in $H^s(\mathbb{R})$. Hence, we have just proved that $y\in C([0,T];H^s(\mathbb{R}))$.

In order to prove $y\in C(\mathbb{R};H^{\frac{2s+3}{8}}(0,T))$, we first write $y=y_1+y_2$, where $$y_1(x,t)\equiv \frac{1}{2\pi}\int_\mathbb{R}e^{i\xi x+i\xi^{4}t}\theta(\xi)\hat{y}_{0}(\xi)d\xi,$$
$$y_2(x,t)\equiv \frac{1}{2\pi}\int_\mathbb{R}e^{i\xi x+i\xi^{4}t}\left(1-\theta(\xi)\right)\hat{y}_{0}(\xi)d\xi,$$ and $\theta$ is a smooth cut-off function satisfying $\theta\equiv 1$ for $|\xi|\le 1$, $0\le \theta\le 1$ for $1<|\xi|<2$, and $\theta\equiv 0$ for $|\xi|\ge 2$. Taking the $j^{\text{th}}$ order time derivative of $y_1$ with $0\le j\le m$, using the definition of $\theta$, Cauchy-Schwarz inequality, and the definition of the Sobolev norm, one deduces that there exists a non-negative constant $c(s,m)\ge 0$ such that
\begin{eqnarray}
\|y_{1}(x)\|_{H^m(0,T)}\leq c(s,m)T^{\frac{1}{2}}\|y_{0}\|_{H^{s}(\mathbb{R})}\label{2.14},
\end{eqnarray}
at first for all $m\in\mathbb{N}$, and then by interpolation for all $m\ge 0$.

In order to obtain a similar estimate for $y_2$, we split it into two terms and write $y_2=y_{2,1}+y_{2,2}$,  where
$$y_{2,1}(x,t)\equiv \frac{1}{2\pi}\int_{-\infty}^{-1}e^{i\xi x+i\xi^{4}t}\left(1-\theta(\xi)\right)\hat{y}_{0}(\xi)d\xi$$ and
$$y_{2,2}(x,t)\equiv \frac{1}{2\pi}\int_1^\infty e^{i\xi x+i\xi^{4}t}\left(1-\theta(\xi)\right)\hat{y}_{0}(\xi)d\xi.$$  Let us first consider the integral given by $y_{2,1}$.  We change the variables in this integral by setting $\xi^4=-\tau$ and we define $\displaystyle z^{\frac{1}{4}}$ for $z\in \mathbb{R}_+$ to be the negative real number which is obtained by taking the argument of $z$ as $4\pi$. Then, $y_{2,1}$ can be rewritten as
$$y_{2,1}(x,t)= -\frac{1}{8\pi}\int_{-\infty}^{-1}e^{i(-\tau)^{\frac{1}{4}}x+i\tau t}\left(1-\theta\left((-\tau)^{\frac{1}{4}}\right)\right)\hat{y}_{0}\left((-\tau)^{\frac{1}{4}}\right)\frac{d\tau}{(-\tau)^{\frac{3}{4}}}.$$  The above formula can be thought of as the inverse (temporal) Fourier transform of the function
\[\displaystyle
\hat{y}_{2,1}(x,\tau):=\left\{
  \begin{array}{ll}
    0, &\tau\in[-1,\infty), \\
    -e^{i(-\tau)^{\frac{1}{4}}x}\left(1-\theta\left((-\tau)^{\frac{1}{4}}\right)\right)\hat{y}_{0}\left((-\tau)^{\frac{1}{4}}\right)\frac{1}{4(-\tau)^{\frac{3}{4}}}, &\tau\in(-\infty,-1).
  \end{array}
\right.
\] Note then,
\begin{multline}\label{timeest1}
   \|y_{2,1}(x,\cdot)\|_{H^{\frac{2s+3}{8}}(0,T)}^2\le \|y_{2,1}(x,\cdot)\|_{H^{\frac{2s+3}{8}}(\mathbb{R})}^2=\int_{\mathbb{R}}(1+\tau^2)^{\frac{2s+3}{8}}|\hat{y}_{2,1}(x,\tau)|^2d\tau \\
  = \int_{-\infty}^{-1}(1+\tau^2)^{\frac{2s+3}{8}}\left(1-\theta\left((-\tau)^{\frac{1}{4}}\right)\right)^2\left|\hat{y}_{0}\left((-\tau)^{\frac{1}{4}}\right)\right|^2\frac{1}{16(-\tau)^{\frac{3}{2}}}d\tau\\
  =-\int_{-\infty}^{-1}(1+\xi^8)^{\frac{2s+3}{8}}\left(1-\theta\left(\xi\right)\right)^2\left|\hat{y}_{0}\left(\xi\right)\right|^2\frac{1}{4\xi^3}d\xi
  \le \int_{-\infty}^{-1}(1+\xi^8)^{\frac{2s+3}{8}}\left|\hat{y}_{0}\left(\xi\right)\right|^2\frac{1}{4\xi^3}d\xi\\
  \lesssim \int_{\mathbb{R}}(1+\xi^2)^{s}\left|\hat{y}_{0}\left(\xi\right)\right|^2d\xi=\|y_0\|_{H^s(\mathbb{R})}^2.
\end{multline} The same estimate is also true for $y_{2,2}$ by similar arguments. Now, using these two estimates together with \eqref{2.14}, one obtains
$$\|y(x,\cdot)\|_{H^{\frac{2s+3}{8}}(0,T)}\le c_s(1+T^{\frac{1}{2}})\|y_{0}\|_{H^{s}(\mathbb{R})}$$ for $s\ge -\frac{3}{2}$. Continuity of the map $x\rightarrow y(x,\cdot)$ can be shown by using the dominated convergence theorem again.
Differentiating the problem \eqref{w1.1}-\eqref{w1.2} in $x$ and repeating the above analysis with initial data $y_0'\in H^{s-1}(\mathbb{R})$, we deduce that $y_x\in C(\mathbb{R};H^{\frac{2s+1}{8}}(0,T))$ if $s\ge -\frac{1}{2}$.
\end{proof}

\subsection{Boundary data-to-solution operator}\label{bdrtosol} In this section, we consider the biharmonic Schrödinger equation with zero initial datum and inhomogeneous Dirichlet-Neumann boundary inputs:
 \begin{align}
  &iz_{t}+\partial_x^4z=0, \quad (x,t)\in\mathbb{R}_+\times (0,T'),\label{3.1}\\
  &z(x,0)=0, \quad x\in\mathbb{R}_+,\label{3.2}\\
  &z(0,t)=h_{0}(t), \quad  t\in(0,T'),\label{3.3} \\
  &z_{x}(0,t)=h_{1}(t), \quad  t\in(0,T').\label{3.4}
  \end{align}
\subsubsection*{Compatibility conditions} In order for solutions to be continuous at the space-time corner point $(x,t)=(0,0)$, we find a set of necessary (compability) conditions based on the analysis of traces.  Suppose that $s\ge 0$, $h_0\in H^{\frac{2s+3}{8}}(\mathbb{R})$ and $h_1\in H^{\frac{2s+1}{8}}(\mathbb{R})$.  Note that if $s\in \left(\frac{1}{2},\frac{9}{2}\right)$, then $\frac{2s+3}{8}>\frac{1}{2}$, and therefore $h_0(0)$ is well-defined, but $z(0,0)=0$, and hence we must have $h_0(0)=0.$ If $s\in\left(\frac{9}{2},\frac{17}{2}\right)$, then $\frac{2s+3}{8}>\frac{3}{2}$, and therefore $h_0(0), h_0'(0)$ are both well-defined, but $z_x(0,0)=0$, and hence we must have $h_0'(0)=0$ in addition to $h_0(0)=0.$ More generally, if $s\in \left(\frac{1}{2}+4(j-1),\frac{1}{2}+4j\right)$ for some  $j\ge 1$, then using also the main equation, we deduce that $\partial_t^kh_0(0)=0$ is a necessary condition for all $0\le k\le j-1$.  By similar arguments, we deduce that if $s\in \left(\frac{3}{2}+4(l-1),\frac{3}{2}+4l\right)$ for some  $l\ge 1$, then $\partial_t^kh_1(0)=0$ is a necessary condition for all $0\le k\le l-1$.
\subsubsection*{Wellposedness}
We prove the following wellposedness result for the half-line problem \eqref{3.1}-\eqref{3.4}.
\begin{lem}\label{lem239}
  Let $s\ge 0$, $s\neq 4j+\frac{1}{2},\,4j+\frac{3}{2}$ for $j\in \mathbb{N}$, $h_0\in H^{\frac{2s+3}{8}}(\mathbb{R})$ and $h_1\in H^{\frac{2s+1}{8}}(\mathbb{R})$, both of which vanish for $t\in \mathbb{R}-(0,T')$ and satisfy the necessary compatibility conditions at $(x,t)=(0,0)$.  Then the solution of \eqref{3.1}-\eqref{3.4}, denoted $z(t)=S_{b}([h_0,h_1])(t)$, satisfies $z\in C([0,T'];H^s(\mathbb{R}_+))\cap C(\mathbb{R}_+;H^{\frac{2s+3}{8}}(0,T'))$ with the estimate
  \begin{multline}\sup_{t\in [0,T']}\|z(\cdot,t)\|_{H^s(\mathbb{R}_+)}+\sup_{x\in \mathbb{R}_+}\|z(x,\cdot)\|_{H^{\frac{2s+3}{8}}(0,T')}\\
  \lesssim \|h_0\|_{H^{\frac{2s+3}{8}}(\mathbb{R})}+\left(1+T'^{\frac{1}{2}}\right)\|h_1\|_{H^{\frac{2s+1}{8}}(\mathbb{R})}.\end{multline}
  Moreover, $z$ has the additional regularity $z_x\in C(\mathbb{R}_+;H^{\frac{2s+1}{8}}(0,T'))$.
\end{lem}
\begin{proof}
Using the representation formula \eqref{qDeformednew2}, the solution of \eqref{3.1}-\eqref{3.4} is given by
  \begin{multline}\label{zsol}
{z}(x,t)=-\frac{1}{\pi}\int_{\partial D_1^+} e^{ikx+ik^4t}\left[k^3(1-i)H_0(-ik^4,T')+k^2\left(1-i\right)H_1(-ik^4,T')\right]dk\\
-\frac{1}{\pi}\int_{\partial D_2^+} e^{ikx+ik^4t}\left[k^3(1+i)H_0(-ik^4,T')-k^2(1+i)H_1(-ik^4,T')\right]dk,
\end{multline} where $H_0$ and $H_1$ are defined in terms of the formula \eqref{Gj} with data $h_0$ and $h_1$, respectively.
\subsubsection*{Space estimate}
Note that $\partial D_1^+ = \gamma_1\cup \gamma_2$ and $\partial D_2^+ = \gamma_3\cup \gamma_4$ (see Figure \ref{gammas}), where
\begin{eqnarray*}
\gamma_{1}(k):=k+ik, \ \ 0\leq k<\infty, \\
\gamma_{2}(k):=-ik, \ \ -\infty<k\leq 0, \\
\gamma_{3}(k):=-k+ik, \ \ 0\le k< \infty, \\
\gamma_{4}(k):=k, \ \ -\infty<k\leq 0. \\
\end{eqnarray*}

\begin{figure}[h]
  \centering
   \includegraphics[scale=.75]{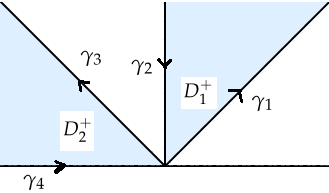}
  \caption{Partitioning the boundary}\label{gammas}
\end{figure}

Therefore, \eqref{zsol} can be rewritten as
  \begin{multline}\label{zsol2}
{z}(x,t)=-\frac{4}{\pi}\int_{0}^\infty e^{-kx+ikx-4ik^4t}\left[(i-1)k^3H_0(4ik^4,T')+ik^2H_1(4ik^4,T')\right]dk\\
+\frac{1-i}{\pi}\int_{0}^\infty e^{-kx+ik^4t}\left[k^3H_0(-ik^4,T')-ik^2H_1(-ik^4,T')\right]dk\\
+\frac{4}{\pi}\int_{0}^\infty e^{-kx-ikx-4ik^4t}\left[(1+i)k^3H_0(4ik^4,T')+ik^2H_1(4ik^4,T')\right]dk\\
+\frac{(1+i)}{\pi}\int_{0}^\infty e^{-ikx+ik^4t}\left[k^3H_0(-ik^4,T')+k^2H_1(-ik^4,T')\right]dk=:\sum_{j=1}^8z_j(x,t).
\end{multline}

We will start by considering the first component of \eqref{zsol2}, which is given by
\begin{equation}\label{z1}z_1(x,t)=-\frac{4}{\pi}\int_{0}^\infty e^{-kx+ikx-4ik^4t}\left[(i-1)k^3H_0(4ik^4,T')\right]dk.\end{equation}

%

Let us first consider the case $s=0$.  In this case, by using the change of variables $\tau=-4k^4$, we get
\begin{multline}\label{z1est0}\|z_1(\cdot,t)\|_{L^2(\mathbb{R}_+)}^2=\frac{16}{\pi^2}\int_{0}^{\infty}\left|\int_{0}^\infty e^{-kx+ikx-4ik^4t}\left[(i-1)k^3H_0(4ik^4,T')\right]dk\right|^{2}dx\\
\le \frac{32}{\pi^2}\int_{0}^{\infty}\left(\int_{0}^\infty e^{-kx}k^3|H_0(4ik^4,T')|dk\right)^{2}dx\le \frac{32}{\pi}\int_{0}^\infty k^6|H_0(4ik^4,T')|^2dk\\
=\frac{32}{\pi}\int_{0}^\infty k^6|\hat{h}_0(-4k^4)|^2dk=\frac{1}{\sqrt{2}\pi}\int_{-\infty}^{0} (-\tau)^\frac{3}{4}|\hat{h}_0(\tau)|^2d\tau\\
\le \frac{1}{\sqrt{2}\pi}\int_{-\infty}^\infty \left(1+{\tau}^2\right)^\frac{3}{8}|\hat{h}_0(\tau)|^2d\tau,\end{multline} where the second inequality above is a property of the Laplace transform (see for example \cite[Lemma 3.2]{fokas}).  Taking the square root in \eqref{z1est0}, we obtain \begin{equation}\label{z1est1}
                     \|z_1(\cdot,t)\|_{L^2(\mathbb{R}_+)}\le \frac{1}{2^{\frac{1}{4}}\sqrt{\pi}}\|h_0\|_{H^{\frac{3}{8}}(\mathbb{R})}.
                   \end{equation}
Similarly,

\begin{multline}\label{z2est0}\|z_2(\cdot,t)\|_{L^2(\mathbb{R}_+)}^2=\frac{16}{\pi^2}\int_{0}^{\infty}\left|\int_{0}^\infty e^{-kx+ikx-4ik^4t}\left[ik^2H_1(4ik^4,T')\right]dk\right|^{2}dx\\
\le \frac{16}{\pi^2}\int_{0}^{\infty}\left(\int_{0}^\infty e^{-kx}k^2|H_1(4ik^4,T')|dk\right)^{2}dx\le \frac{16}{\pi}\int_{0}^\infty k^4|H_1(4ik^4,T')|^2dk\\
=\frac{16}{\pi}\int_{0}^\infty k^4|\hat{h}_1(-4k^4)|^2dk=\frac{1}{\sqrt{2}\pi}\int_{-\infty}^{0} (-\tau)^\frac{1}{4}|\hat{h}_1(\tau)|^2d\tau\\
\le \frac{1}{\sqrt{2}\pi}\int_{-\infty}^\infty \left(1+{\tau}^2\right)^\frac{1}{8}|\hat{h}_1(\tau)|^2d\tau.\end{multline} Taking the square root in \eqref{z2est0}, we obtain \begin{equation}\label{z2est1}
                     \|z_2(\cdot,t)\|_{L^2(\mathbb{R}_+)}\le \frac{1}{2^{\frac{1}{4}}\sqrt{\pi}}\|h_1\|_{H^{\frac{1}{8}}(\mathbb{R})}.
                   \end{equation}

Now, let $s=m\in \mathbb{Z}_+$ and $1\le j\le m$. Then, \begin{multline}\label{z1est2}\|\partial_x^jz_1(\cdot,t)\|_{L^2(\mathbb{R}_+)}^2=\frac{16}{\pi^2}\int_{0}^{\infty}\left|\int_{0}^\infty e^{-kx+ikx-4ik^4t}\left[(i-1)^{j+1}k^{3+j}H_0(4ik^4,T')\right]dk\right|^{2}dx\\
\le \frac{2^{j+5}}{\pi^2}\int_{0}^{\infty}\left(\int_{0}^\infty e^{-kx}k^{3+j}|H_0(4ik^4,T')|dk\right)^{2}dx\le \frac{2^{j+5}}{\pi}\int_{0}^\infty k^{2j+6}|H_0(4ik^4,T')|^2dk\\
=\frac{2^{j+5}}{\pi}\int_{0}^\infty k^{2j+6}|\hat{h}_0(-4k^4)|^2dk=\frac{1}{\sqrt{2}\pi}\int_{-\infty}^{0} (-\tau)^\frac{2j+3}{4}|\hat{h}_0(\tau)|^2d\tau\\
\le \frac{1}{\sqrt{2}\pi}\int_{-\infty}^\infty \left(1+{\tau}^2\right)^\frac{2j+3}{8}|\hat{h}_0(\tau)|^2d\tau,\end{multline} which gives
\begin{equation}\label{z1est3}
                     \|z_1(\cdot,t)\|_{H^m(\mathbb{R}_+)}\le \frac{1}{2^{\frac{1}{4}}\sqrt{\pi}}\|h_0\|_{H^{\frac{2m+3}{8}}(\mathbb{R})}.
                   \end{equation}
Similarly,
\begin{multline}\label{z2est2}\|\partial_x^jz_2(\cdot,t)\|_{L^2(\mathbb{R}_+)}^2=\frac{16}{\pi^2}\int_{0}^{\infty}\left|\int_{0}^\infty e^{-kx+ikx-4ik^4t}\left[(i-1)^{j}ik^{2+j}H_1(4ik^4,T')\right]dk\right|^{2}dx\\
\le \frac{2^{j+4}}{\pi^2}\int_{0}^{\infty}\left(\int_{0}^\infty e^{-kx}k^{2+j}|H_1(4ik^4,T')|dk\right)^{2}dx\le \frac{2^{j+4}}{\pi}\int_{0}^\infty k^{2j+4}|H_1(4ik^4,T')|^2dk\\
=\frac{2^{j+4}}{\pi}\int_{0}^\infty k^{2j+4}|\hat{h}_1(-4k^4)|^2dk=\frac{1}{\sqrt{2}\pi}\int_{-\infty}^{0} (-\tau)^\frac{2j+1}{4}|\hat{h}_1(\tau)|^2d\tau\\
\le \frac{1}{\sqrt{2}\pi}\int_{-\infty}^\infty \left(1+{\tau}^2\right)^\frac{2j+1}{8}|\hat{h}_1(\tau)|^2d\tau,\end{multline} which gives
\begin{equation}\label{z2est3}
                     \|z_2(\cdot,t)\|_{H^m(\mathbb{R}_+)}\le \frac{1}{2^{\frac{1}{4}}\sqrt{\pi}}\|h_1\|_{H^{\frac{2m+1}{8}}(\mathbb{R})}.
                   \end{equation}
Now, by interpolation between \eqref{z1est1} and \eqref{z1est3}, we obtain
\begin{equation}\label{z1est4}
                     \|z_1(\cdot,t)\|_{H^s(\mathbb{R}_+)}\lesssim\|h_0\|_{H^{\frac{2s+3}{8}}(\mathbb{R})}
                   \end{equation} for all $s\ge 0$, including the non-integer values of $s$. Similarly, by interpolating between \eqref{z2est1} and \eqref{z2est3}, we get \begin{equation}\label{z2est4}
                     \|z_2(\cdot,t)\|_{H^s(\mathbb{R}_+)}\lesssim\|h_1\|_{H^{\frac{2s+1}{8}}(\mathbb{R})}
                   \end{equation} for $s\ge 0$. Applying the same arguments also to $z_j$ for $3\le j\le 8$, we get
\begin{equation}\label{z34est4}
                     \|z_3(\cdot,t)\|_{H^s(\mathbb{R}_+)}\lesssim\|h_0\|_{H^{\frac{2s+3}{8}}(\mathbb{R})}\text{ and } \|z_4(\cdot,t)\|_{H^s(\mathbb{R}_+)}\lesssim\|h_1\|_{H^{\frac{2s+1}{8}}(\mathbb{R})},
                   \end{equation}
\begin{equation}\label{z56est4}
                     \|z_5(\cdot,t)\|_{H^s(\mathbb{R}_+)}\lesssim\|h_0\|_{H^{\frac{2s+3}{8}}(\mathbb{R})}\text{ and } \|z_6(\cdot,t)\|_{H^s(\mathbb{R}_+)}\lesssim\|h_1\|_{H^{\frac{2s+1}{8}}(\mathbb{R})},
                   \end{equation}
\begin{equation}\label{z78est4}
                     \|z_7(\cdot,t)\|_{H^s(\mathbb{R}_+)}\lesssim\|h_0\|_{H^{\frac{2s+3}{8}}(\mathbb{R})}\text{ and } \|z_8(\cdot,t)\|_{H^s(\mathbb{R}_+)}\lesssim\|h_1\|_{H^{\frac{2s+1}{8}}(\mathbb{R})}.
                   \end{equation}
Combining \eqref{z1est4}-\eqref{z78est4}, we deduce that
\begin{equation}\label{zCnt1}
  \|z(\cdot,t)\|_{H^s(\mathbb{R}_+)}\lesssim \|h_0\|_{H^{\frac{2s+3}{8}}(\mathbb{R})}+\|h_1\|_{H^{\frac{2s+1}{8}}(\mathbb{R})}.
\end{equation} Continuity in $t$ can easily be justified by means of the dominated convergence theorem. Therefore we just proved that $z\in C([0,T'];H^s(\mathbb{R}))$ under the given assumptions on $s$, $h_0$ and $h_1$.
\subsubsection*{Time estimate}
We can rewrite \eqref{zsol} as
  \begin{multline}\label{zsol3}
{z}(x,t)=-\frac{1}{\pi}\int_{0}^{(1+i)\infty} e^{ikx+ik^4t}\left[k^3(1-i)\hat{h}_0(k^4)+k^2\left(1-i\right)\hat{h}_1(k^4)\right]dk\\
-\frac{1}{\pi}\int_{i\infty}^{0} e^{ikx+ik^4t}\left[k^3(1-i)\hat{h}_0(k^4)+k^2\left(1-i\right)\hat{h}_1(k^4)\right]dk\\
-\frac{1}{\pi}\int_{0}^{(-1+i)\infty} e^{ikx+ik^4t}\left[k^3(1+i)\hat{h}_0(k^4)-k^2(1+i)\hat{h}_1(k^4)\right]dk\\
-\frac{1}{\pi}\int_{-\infty}^{0} e^{ikx+ik^4t}\left[k^3(1+i)\hat{h}_0(k^4)-k^2(1+i)\hat{h}_1(k^4)\right]dk=:\sum_{j=1}^8z_j(x,t).
\end{multline} We will start with considering the first component of \eqref{zsol3}, which is given by
$$z_1(x,t)=-\frac{1}{\pi}\int_{0}^{(1+i)\infty} e^{ikx+ik^4t}\left[k^3(1-i)\hat{h}_0(k^4)\right]dk.$$  By using the change of variables $\tau=k^4$, we get
$$z_1(x,t)=\frac{1-i}{4\pi}\int_{-\infty}^{0} e^{i\tau^\frac{1}{4}x+it\tau}\hat{h}_0(\tau)d\tau,$$ where $z^{\frac{1}{4}}$ for $z\in\mathbb{R}_{-}$ is defined by using the argument equal to $\pi$ so that $\tau^{\frac{1}{4}}\in \gamma_1$ for $\tau\in (-\infty,0)$.  Note that we can regard $z_1(x,\cdot)$ as the inverse Fourier transform of the function \[\displaystyle
\hat{z}_{1}(x,\tau):=\left\{
  \begin{array}{ll}
    0, &\tau\in[0,\infty), \\
    \frac{1}{2}e^{i\tau^\frac{1}{4}x}\left[(1-i)\hat{h}_0(\tau)\right], &\tau\in(-\infty,0),
  \end{array}
\right.
\] from which it follows that
\begin{multline}\label{z1257}\|z_1(x,\cdot)\|_{H^{\frac{2s+3}{8}}(0,T')}^2\le \|z_1(x,\cdot)\|_{H^{\frac{2s+3}{8}}(\mathbb{R})}^2\\
\le \frac{1}{2}\int_{\mathbb{R}}(1+\tau^2)^{\frac{2s+3}{8}}|\hat{h}_0(\tau)|^2d\tau=\frac{1}{2}\|h_0\|_{H^{\frac{2s+3}{8}}(\mathbb{R})}^2\end{multline} taking into account that $\text{Im}\left(\tau^{\frac{1}{4}}\right)\ge 0$ so that $|e^{i\tau^\frac{1}{4}x}|\le 1$.

Now, we will estimate the term
$$z_2(x,t)=-\frac{1}{\pi}\int_{0}^{(1+i)\infty} e^{ikx+ik^4t}\left[k^2\left(1-i\right)\hat{h}_1(k^4)\right]dk.$$

In order to do this, we write $z_2=z_{2,1}+z_{2,2}$ where
$$z_{2,1}(x,t)=-\frac{1}{\pi}\int_{0}^{(1+i)\infty}e^{ikx+ik^4t} \theta(k)\left[k^2\left(1-i\right)\hat{h}_1(k^4)\right]dk,$$
$$z_{2,2}(x,t)=-\frac{1}{\pi}\int_{0}^{(1+i)\infty}e^{ikx+ik^4t}\left(1-\theta(k)\right)\left[k^2\left(1-i\right)\hat{h}_1(k^4)\right]dk,$$ and $\theta$ is a smooth cut-off function satisfying $\theta\equiv 1$ for $|k|\le \sqrt{2}$, $0\le \theta\le 1$ for $\sqrt{2}<|k|<2\sqrt{2}$, and $\theta\equiv 0$ for $|k|\ge 2\sqrt{2}$ for $k\in \gamma_1$. By using the properties of $\theta$, we can rewrite $z_{2,1}$ as
$$z_{2,1}(x,t)=-\frac{1}{\pi}\int_{0}^{2(1+i)}e^{ikx+ik^4t}\theta(k)\left[k^2\left(1-i\right)\hat{h}_1(k^4)\right]dk.$$ Using the same change of variables $\tau=k^4$ as before, we get
$$z_{2,1}(x,t)=\frac{\left(1-i\right)}{4\pi}\int_{-2^6}^{0}e^{i\tau^{\frac{1}{4}}x+it\tau}\theta\left(\tau^{\frac{1}{4}}\right)\hat{h}_1(\tau)\frac{d\tau}{\tau^{\frac{1}{4}}}.$$
The $j^{\text{th}}$ order time derivative of $z_{2,1}(x,\cdot)$ is then written
$$\partial_t^jz_{2,1}(x,t)=\frac{\left(1-i\right)}{4\pi}\int_{-2^6}^{0}e^{i\tau^{\frac{1}{4}}x+it\tau}\theta\left(\tau^{\frac{1}{4}}\right)(i\tau)^j\hat{h}_1(\tau)\frac{d\tau}{\tau^{\frac{1}{4}}}.$$

We will estimate the above identity in two cases: (i) $s<\frac{3}{2}$, (ii) $s\ge \frac{3}{2}$. If $s<\frac{3}{2}$, then $\frac{2s+1}{8}<\frac{1}{2}$ and $2j-\frac{1}{2}-\frac{2s+1}{8}>-1$ for all $j\in \mathbb{N}$.  Therefore, by using the Cauchy-Schwarz inequality:
\begin{multline}
  |\partial_t^jz_{2,1}(x,t)|\le \frac{\sqrt{2}}{4\pi}\left(\int_{-2^6}^{0}(1+\tau^2)^{-\frac{2s+1}{8}}\tau^{2j-\frac{1}{2}}d\tau\right)^{\frac{1}{2}}
  \left(\int_{-2^6}^{0}(1+\tau^2)^{\frac{2s+1}{8}}|\hat{h}_1(\tau)|^2d\tau\right)^{\frac{1}{2}}\\
  \le \frac{\sqrt{2}}{4\pi}\left(\int_{-2^6}^{0}\tau^{2j-\frac{1}{2}-\frac{2s+1}{8}}d\tau\right)^{\frac{1}{2}}\|h_1\|_{H^{\frac{2s+1}{8}}(\mathbb{R})}=c_{j,s}\|h_1\|_{H^{\frac{2s+1}{8}}(\mathbb{R})}.
\end{multline}
Note that $c_{j,s}<\infty$ in the above estimate since $2j-\frac{1}{2}-\frac{2s+1}{8}>-1$. Now, consider the case $s\ge \frac{3}{2}$.  In this case, pick some $s'<3/2.$ As before, we have
$$|\partial_t^jz_{2,1}(x,t)|\le c_{j,s'}\|h_1\|_{H^{\frac{2s'+1}{8}}(\mathbb{R})}.$$ But since in particular $s'<s$, we have $\|h_1\|_{H^{\frac{2s'+1}{8}}(\mathbb{R})}\le \|h_1\|_{H^{\frac{2s+1}{8}}(\mathbb{R})}$.  Therefore, we deduce that $$|\partial_t^jz_{2,1}(x,t)|\le c_{j,s'}\|h_1\|_{H^{\frac{2s+1}{8}}(\mathbb{R})}.$$
From the definition of the Sobolev norm, one obtains
\begin{eqnarray}
\|z_{2,1}(x)\|_{H^m(0,T')}\leq c_{m,s'}T'^{\frac{1}{2}}\|h_1\|_{H^{\frac{2s+1}{8}}(\mathbb{R})}
\end{eqnarray}
at first for all $m\in\mathbb{N}$, and then by interpolation for all $m\ge 0$.  In particular, by choosing $m=\frac{2s+3}{8}$ one has
\begin{eqnarray}\label{z22est1}
\|z_{2,1}(x)\|_{H^{\frac{2s+3}{8}}(0,T')}\leq c_sT'^{\frac{1}{2}}\|h_1\|_{H^{\frac{2s+1}{8}}(\mathbb{R})}.
\end{eqnarray} for $s\ge -\frac{3}{2}$.

Now, let us estimate $z_{2,2}(x,\cdot)$. By using the definition of $\theta$, we can rewrite this function as
$$z_{2,2}(x,t)=-\frac{1}{\pi}\int_{1+i}^{(1+i)\infty}e^{ikx+ik^4t}\left(1-\theta(k)\right)\left[k^2\left(1-i\right)\hat{h}_1(k^4)\right]dk.$$ Using the same change of variables $\tau=k^4$ with argument of $\tau$ equal to $\pi$, we have
$$z_{2,2}(x,t)=\frac{\left(1-i\right)}{4\pi}\int_{-\infty}^{-4}e^{i\tau^{\frac{1}{4}}x+it\tau}\left(1-\theta\left(\tau^{\frac{1}{4}}\right)\right)\hat{h}_1(\tau)\frac{d\tau}{\tau^{\frac{1}{4}}}.$$
Note that, we can regard $z_{2,2}(x,\cdot)$ as the inverse Fourier transform of the function \[\displaystyle
\hat{z}_{2,2}(x,\tau):=\left\{
  \begin{array}{ll}
    0, &\tau\in[-4,\infty), \\
    \frac{1}{2}e^{i\tau^\frac{1}{4}x}\left(1-\theta\left(\tau^{\frac{1}{4}}\right)\right)(1-i)\hat{h}_1(\tau)\frac{1}{\tau^{\frac{1}{4}}}, &\tau\in(-\infty,-4),
  \end{array}
\right.
\] from which it follows that
\begin{multline}\label{z22est2}\|z_{2,2}(x,\cdot)\|_{H^{\frac{2s+3}{8}}(0,T')}^2\le \|z_{2,2}(x,\cdot)\|_{H^{\frac{2s+3}{8}}(\mathbb{R})}^2=\int_{\mathbb{R}}(1+\tau^2)^{\frac{2s+3}{8}}|\hat{z}_{2,2}(\tau)|^2d\tau\\
\le \frac{1}{2}\int_{-\infty}^{-4}(1+\tau^2)^{\frac{2s+3}{8}}|\tau|^{-\frac{1}{2}}|\hat{h}_1(\tau)|^2d\tau\lesssim\int_{\mathbb{R}}(1+\tau^2)^{\frac{2s+1}{8}}|\hat{h}_1(\tau)|^2d\tau=\|h_1\|_{H^{\frac{2s+1}{8}}(\mathbb{R})}^2.\end{multline}
Combining \eqref{z22est1} and \eqref{z22est2}, we conclude that
\begin{eqnarray}\label{z22est3}
\|z_{2}(x)\|_{H^{\frac{2s+3}{8}}(0,T')}\lesssim \left(1+T'^{\frac{1}{2}}\right)\|h_1\|_{H^{\frac{2s+1}{8}}(\mathbb{R})}.
\end{eqnarray}
Applying similar arguments that we used for $z_1$ and $z_2$ also to $z_j$ for $3\le j\le 8$ and combining the relevant estimates, we deduce that
\begin{eqnarray}\label{ztestson}
\|z(x)\|_{H^{\frac{2s+3}{8}}(0,T')}\lesssim \|h_0\|_{H^{\frac{2s+3}{8}}(\mathbb{R})}+\left(1+T'^{\frac{1}{2}}\right)\|h_1\|_{H^{\frac{2s+1}{8}}(\mathbb{R})}.
\end{eqnarray}
Continuity in $x$ can easily be justified by means of the dominated convergence theorem. Therefore we just proved that $z\in C(\mathbb{R}_+;H^{\frac{2s+3}{8}}(0,T'))$ under the given assumptions on $h_0$ and $h_1$. The time estimate for $z_x$ can be proven using arguments as in \eqref{z1257} after differentiating in $x$. Note that this will bring an extra factor of $|\tau|^{1/4}$ into the integrals in the estimates.  Therefore one will obtain $z_x\in C(\mathbb{R}_+;H^{\frac{2s+1}{8}}(0,T'))$.
\end{proof}
\subsection{Representation formula - revisited}Let $s\in \left[0,\frac{9}{2}\right)\setminus\left\{\frac{1}{2},\frac{3}{2}\right\}$.
Let $(\cdot)^*$ denote a bounded extension operator from $H^s(\mathbb{R}_+)$ into $H^s(\mathbb{R})$. Existence of such an extension operator is guaranteed by the definition of the Sobolev space and the associated Sobolev norm on the half-line:
$$H^s(\mathbb{R}_+)=\{\varphi:\mathbb{R}_+ \rightarrow \mathbb{C}\,|\,\exists \psi\in H^s(\mathbb{R}) \text{ s.t. } \psi|_{\mathbb{R}_+} = \varphi\},$$ where
$$\|\varphi\|_{H^s(\mathbb{R}_+)}=\inf_{\substack{{\psi\in H^s(\mathbb{R})} \\ {\psi|_{\mathbb{R}_+}} = \varphi}}\|\psi\|_{H^s(\mathbb{R})}.$$  Now, given $q_0\in H^s(\mathbb{R}_+)$, we let $q_0^*\in H^s(\mathbb{R})$ be the extension of $q$ with respect to the fixed extension operator $(\cdot)^*$ just defined. Note that by the boundedness of this operator we have $$\|q_0^*\|_{H^s(\mathbb{R})}\lesssim \|q_0\|_{H^s(\mathbb{R}^+)}.$$

Therefore, $y(t)=S_\mathbb{R}(t)q_0^*$ solves the problem
\begin{equation}\label{yext}
iy_t + \partial_x^4y=0,y(0,t)=q_0^*(x), x\in \mathbb{R}, t\in (0,T),
\end{equation}  where $S_\mathbb{R}(t)$ is the evolution operator for the free linear biharmonic Schrödinger equation given in Lemma \ref{Wrty0}.  Similarly, given $f\in L^1(0,T;H^s(\mathbb{R}_+))$, let $f^*\in L^1(0,T;H^s(\mathbb{R}))$ be the extension of $f$ in the spatial variable with respect to the extension operator $(\cdot)^*$. Then the solution of the non-homogeneous Cauchy problem
\begin{equation}\label{nonhomprob}
iw_t+\partial_x^4 w=f^*, w(x,0)=0, x\in \mathbb{R}, t\in (0,T)\end{equation} can be written as $$w(t)=-i\int_0^tS_\mathbb{R}(t-\tau)f^*(\tau)d\tau.$$

For $j=0,1$, we set $$a_j(t):=\left.\partial_x^jS_\mathbb{R}(t)q^*_0\right|_{x=0} \text{ and } b_j(t):=\left.-i\partial_x^j\int_0^tS_\mathbb{R}(t-\tau)f^*(\tau)d\tau\right|_{x=0}.$$ Note that these traces exist by Lemma \ref{Wrty0} as elements of $H^{\frac{2s+3}{8}}(0,T)$ for $j=0$ and $H^{\frac{2s+1}{8}}(0,T)$  for $j=1$. See also the analog argument for the classical Schrödinger equation given in \cite[Prop. 3.7 (iii)]{bona}.

\subsubsection*{Compatibility conditions} In order for solutions to be continuous at the space-time corner point $(x,t)=(0,0)$, we can find the necessary (compability) conditions based on the analysis of traces again.  Suppose that $s\in \left[0,\frac{9}{2}\right)\setminus\left\{\frac{1}{2},\frac{3}{2}\right\}$, $g_0\in H^{\frac{2s+3}{8}}(\mathbb{R})$ and $g_1\in H^{\frac{2s+1}{8}}(\mathbb{R})$.  Note that if $s\in \left(\frac{1}{2},\frac{9}{2}\right)$, then $\frac{2s+3}{8}>\frac{1}{2}$ and therefore $g_0(0)$ is well-defined, but $q(0,0)=q_0(0)$, and hence we must have $g_0(0)=q_0(0) (=a_0(0)).$ By a similar argument, we deduce that if $s\in \left(\frac{3}{2},\frac{9}{2}\right)$, then $g_1(0)=q_0'(0)(=a_1(0))$ is a necessary condition.

On the other hand, if $[p_0,p_1]\in H^{\frac{2s+3}{8}}(0,T)\times H^{\frac{2s+1}{8}}(0,T)$ such that $p_0(0)=p_1(0)=0$, then we set $[p_0,p_1]_e=[[p_0]_e,[p_1]_e]$ to be an extension of $[p_0,p_1]$ to $\mathbb{R}$, which satisfies the compact support condition $[p_0]_e|_{(0,2T)^c}=[p_1]_e|_{(0,2T)^c}=0$, regularity $[p_0,p_1]_e\in H^{\frac{2s+3}{8}}(\mathbb{R})\times H^{\frac{2s+1}{8}}(\mathbb{R})$, and the estimates
$$\|[p_0]_e\|_{H^{\frac{2s+3}{8}}(\mathbb{R})}\lesssim \|p_0\|_{H^{\frac{2s+3}{8}}(0,T)}\text{ and }\|[p_1]_e\|_{H^{\frac{2s+1}{8}}(\mathbb{R})}\lesssim \|p_1\|_{H^{\frac{2s+3}{8}}(0,T)}.$$ For the existence of such extension, see for instance \cite[Lemma 2.1]{BatalOzsari16}.

Therefore, if we define \begin{multline}\label{RepForm}q_e(t):=\left.S_\mathbb{R}(t)q_0^*\right|_{\mathbb{R}_+}-\left.i\int_0^tS_\mathbb{R}(t-\tau)f^*(\tau)d\tau\right|_{\mathbb{R}_+}\\
+S_{b}([g_0-a_0-b_0,g_1-a_1-b_1]_e)(t),\end{multline} then $q={q_e}|_{[0,T)}$ solves \eqref{1.1}-\eqref{1.4}. Note that both $g_0-a_0-b_0$ and $g_1-a_1-b_1$ satisfy the necessary comptability conditions given in Section \ref{bdrtosol} since $g_0(0)=a_0(0)$, $b_0(0)=b_1(0)=0$, and $g_1(0)=a_1(0)$.

Now, we are ready to state the following lemma which follows by combining the space-time estimates proved for the solution generators $S_\mathbb{R}$ and $S_b$ in the previous sections.
\begin{lem} Let $s\in \left[0,\frac{9}{2}\right)\setminus\left\{\frac{1}{2},\frac{3}{2}\right\}$, $q_0\in H^s(\mathbb{R}_+)$, $g_0\in H^{\frac{2s+3}{8}}(0,T)$, $g_1\in H^{\frac{2s+1}{8}}(0,T)$, and $f\in L^1(0,T;H^s(\mathbb{R}_+))$. Suppose also that the initial-boundary data satisfies the necessary compatibility conditions.  Then the following estimate holds true for the solution of \eqref{1.1}-\eqref{1.4}.
  \begin{multline}\|q\|_{C([0,T];H^s(\mathbb{R}_+))}\\
  \lesssim \|q_0\|_{H^s(\mathbb{R}_+)}+\|g_0\|_{H^{\frac{2s+3}{8}}(\mathbb{R}_+)}+\left(1+T^{\frac{1}{2}}\right)\left[\|g_1\|_{H^{\frac{2s+1}{8}}(\mathbb{R}_+)}+\|f\|_{L^1(0,T;H^s(\mathbb{R}_+))}\right].\end{multline}
\end{lem}

\section{Nonlinear model}
\subsection{Local wellposedness for $s>\frac{1}{2}$}
In this section, we assume $s>\frac{1}{2}$ and consider the nonlinear model \eqref{4th.1}-\eqref{4th.4}.
\subsubsection*{Local existence}
Note that a solution of \eqref{4th.1}-\eqref{4th.4} is  a fixed point of the operator $\Psi$ which is given by
\begin{multline}\label{Psi}[\Psi(q)](t):=\left.S_\mathbb{R}(t)q^*_0\right|_{\mathbb{R}_+}-\left.i\int_0^tS_\mathbb{R}(t-\tau)f(q^*(\tau))d\tau\right|_{\mathbb{R}_+}\\
+S_{b}([g_0-a_0-b_0(q^*),g_1-a_1-b_1(q^*)]_e)(t),\end{multline} where
for $j=0,1$, we set $$a_j(t):=\left.\partial_x^jS_\mathbb{R}(t)q^*_0\right|_{x=0} \text{ and } b_j(q^*)(t):=\left.-i\partial_x^j\int_0^tS_\mathbb{R}(t-\tau)f(q^*(\tau))d\tau\right|_{x=0}.$$ We consider the operator $\Psi$ on the Banach space $X_{T_0}:= C([0,T_0];H^s(\mathbb{R}_+))$.  In order to prove the local existence of solutions we will use the Banach fixed point theorem on a closed ball $\overline{B}_R(0)$ of the function space $X_{T_0}$ for appropriately chosen $R>0$ and $T_0\in (0,T]$. Let us first show that $\Psi$ maps $\overline{B}_R(0)$ onto itself for appropriate $R$ and sufficiently small $T_0.$ First of all, we have via Lemma \ref{Wrty0} and the boundedness of the extension operator the estimate
\begin{equation}\label{firstest}
  \|S_\mathbb{R}(t)q^*_0\|_{H^s(\mathbb{R}_+)}\le \|S_\mathbb{R}(t)q^*_0\|_{H^s(\mathbb{R})}=\|q_0^*\|_{H^s(\mathbb{R})}\lesssim \|q_0\|_{H^s(\mathbb{R}_+)},
\end{equation} which gives $$\|S_\mathbb{R}(\cdot)q^*_0\|_{X_{T_0}}\lesssim  \|q_0\|_{H^s(\mathbb{R}_+)}.$$

Let us recall the following lemma, which holds true under the given assumptions (a1) and (a2).  The proof can be found for instance in \cite[Lemma 3.1]{BatalOzsari16}:
\begin{lem}\label{fEst} Let $f(y)=|y|^py$ and $s>\frac{1}{2}$. Then \begin{equation}\label{fEstLem1}\|f(y)\|_{H^s(\mathbb{R})}\lesssim\|y\|_{H^s(\mathbb{R})}^{p+1},\end{equation}
\begin{equation}\|f(y)-f(z)\|_{H^s(\mathbb{R})}\lesssim (\|y\|_{H^s(\mathbb{R})}^{p}+\|z\|_{H^s(\mathbb{R})}^{p})\|y-z\|_{H^s(\mathbb{R})}\end{equation} for $y,z\in H^s(\mathbb{R})$.
\end{lem}
Using Lemma \ref{fEst}, we have
\begin{multline}\left\|-i\int_0^tS_\mathbb{R}(t-\tau)f(q^*(\tau))d\tau\right\|_{X_{T_0}^s}\le \int_0^{T_0}\|f(q^*(\tau))\|_{H^s(\mathbb{R})}d\tau\\
\lesssim \int_0^{T_0}\|q^*(\tau)\|_{H^s(\mathbb{R})}^{p+1}d\tau\lesssim \int_0^{T_0}\|q(\tau)\|_{H^s(\mathbb{R}_+)}^{p+1}d\tau\le T_0\|q\|_{X_{T_0}}^{p+1}.\end{multline}

Similarly,
\begin{multline}\left\|-i\int_0^tS_\mathbb{R}(t-\tau)[f(y^*(\tau))-f(z^*(\tau))]d\tau\right\|_{X_{T_0}}\le \int_0^{T_0}\|f(y^*(\tau))-f(z^*(\tau))\|_{H^s(\mathbb{R})}d\tau\\
\lesssim \int_0^{T_0}(\|y^*(\tau)\|_{H^s(\mathbb{R})}^p+\|z^*(\tau)\|_{H^s(\mathbb{R})}^p)\|y^*(\tau)-z^*(\tau)\|_{H^s(\mathbb{R})}d\tau\\
\lesssim \int_0^{T_0}(\|y(\tau)\|_{H^s(\mathbb{R}_+)}^p+\|z(\tau)\|_{H^s(\mathbb{R}_+)}^p)\|y(\tau)-z(\tau)\|_{H^s(\mathbb{R}_+)}d\tau\\
\lesssim {T_0}(\|y\|_{X_{T_0}}^p+\|z\|_{X_{T_0}}^p)\|y-z\|_{X_{T_0}}.\end{multline}

The last term in \eqref{Psi} is estimated as follows
\begin{multline}\label{Wbhgp}
\|S_{b}([g_0-a_0-b_0(q^*),g_1-a_1-b_1(q^*)]_e)(t)\|_{X_{T_0}}\\
\lesssim\|[g_0-a_0-b_0(q^*)]_e\|_{H^{\frac{2s+3}{8}}(\mathbb{R})}+\left(1+T_0^\frac{1}{2}\right)\|[g_1-a_1-b_1(q^*)]_e\|_{H^{\frac{2s+1}{8}}(\mathbb{R})}\\
\lesssim \|g_0\|_{H^{\frac{2s+3}{8}}(0,T_0)}+\|a_0\|_{H^{\frac{2s+3}{8}}(0,T_0)}+\|b_0(q^*)\|_{H^{\frac{2s+3}{8}}(0,T_0)}\\
+\left(1+T_0^\frac{1}{2}\right)\left[\|g_1\|_{H^{\frac{2s+1}{8}}(0,T_0)}+\|a_1\|_{H^{\frac{2s+1}{8}}(0,T_0)}+\|b_1(q^*)\|_{H^{\frac{2s+1}{8}}(0,T_0)}\right].
\end{multline}

Note that \begin{multline}\label{gEst}\|a_1\|_{H^{\frac{2s+1}{8}}(0,T_0)}=\|\partial_x S_\mathbb{R}(t)q^*_0|_{x=0}\|_{H^{\frac{2s+1}{8}}(0,T_0)}\le \sup_{x\in\mathbb{R}_+}\|\partial_x S_\mathbb{R}(t)q^*_0\|_{H^{\frac{2s+1}{8}}(0,T_0)}\\
\le \left\|\frac{d}{dx}q_0^*\right\|_{H^{s-1}(\mathbb{R})}\le \|q_0^*\|_{H^{s}(\mathbb{R})}\lesssim \|q_0\|_{H^s(\mathbb{R}_+)}.\end{multline}

In \eqref{gEst}, the second inequality follows from  the fact that $\partial_x S_\mathbb{R}(t)q^*_0$ is a solution of the linear biharmonic Schrödinger equation on $\mathbb{R}$ with initial condition $\displaystyle\frac{d}{dx}q_0^*$.

Similarly,
\begin{multline}\label{pEst}\|b_1(q^*)\|_{H^{\frac{2s+1}{8}}(0,T_0)}=\left\|-i\partial_x\int_0^tS_\mathbb{R}(t-\tau)f(q^*(\tau))d\tau|_{x=0}\right\|_{H^{\frac{2s+1}{8}}(0,T_0)}\\
\le \sup_{x\in\mathbb{R}_+}\left\|-i\partial_x\int_0^tS_\mathbb{R}(t-\tau)f(q^*(\tau))d\tau\right\|_{H^{\frac{2s+1}{8}}(0,T_0)}\\
\lesssim \|\partial_xf(q^*)\|_{L^1(0,T_0;H^{s-1}(\mathbb{R}))}\le \|f(q^*)\|_{L^1(0,T_0;H^{s}(\mathbb{R}))}\lesssim T_0\|q\|_{X_{T_0}}^{p+1}\end{multline}

and \begin{equation}\|b_1(y^*)-b_1(z^*)\|_{H^{\frac{2s+1}{8}}(0,T_0)}\\
\lesssim T_0(\|y\|_{X_{T_0}}^p+\|z\|_{X_{T_0}}^p)\|y-z\|_{X_{T_0}}. \end{equation}

Combining above estimates, we obtain $$\|\Psi(q)\|_{X_{T_0}}\le C\left(\|q_0\|_{H^s(\mathbb{R}_+)}+\|g_0\|_{H^{\frac{2s+3}{8}}(0,T_0)}+\|g_1\|_{H^{\frac{2s+1}{8}}(0,T_0)}+T_0\|q\|_{X_{T_0}}^{p+1}\right).$$

Regarding the differences, again by above estimates, we have $$\|\Psi(y)-\Psi(z)\|_{X_{T_0}}\le C T_0(\|y\|_{X_{T_0}}^p+\|z\|_{X_{T_0}}^p)\|y-z\|_{X_{T_0}}.$$

Now let $A:=C\left(\|q_0\|_{H^s(\mathbb{R}_+)}+\|g_0\|_{H^{\frac{2s+3}{8}}(0,T_0)}+\|g_1\|_{H^{\frac{2s+1}{8}}(0,T_0)}\right)$, $R=2A$ and $T_0$ be small enough that $A+CT_0R^{p+1}<2A$.  Now, if necessary we can choose $T_0$ even smaller so that $\Psi$ becomes a contraction on $\overline{B}_R(0)\subset X_{T_0}$, which is a complete space.  Hence, $\Psi$ must have a unique fixed point in $\overline{B}_R(0)$ when we look for a solution whose lifespan is sufficiently small.
\subsubsection*{Uniqueness}\label{Uniq}Let $q_1,q_2\in X_{T_0}$ be two local solutions.  Then, \begin{multline}q_1(t)-q_2(t)=-i\int_0^tS_\mathbb{R}(t-s)[f(q_1^*(s))-f(q_2^*(s))]ds\\
+S_b(t)([b_0(q_2^*)-b_0(q_2^*),b_1(q_2^*)-b_1(q_2^*)]_e)\end{multline} for a.a. $t\in [0,T_0]$.

Since $s>1/2$,  \begin{multline}\|q_1(t)-q_2(t)\|_{H^s(\mathbb{R}_+)}\le C\int_0^{T_0}\|f(q_1^*(s))-f(q_2^*(s))\|_{H^s(\mathbb{R})}ds\\
+C\|b_0(q_2^*)-b_0(q_1^*)\|_{H^{\frac{2s+3}{8}}(0,T_0)}+C(1+T_0^\frac{1}{2})\|b_1(q_2^*)-b_1(q_1^*)\|_{H^{\frac{2s+1}{8}}(0,T_0)}\\
\le C(1+T_0^\frac{1}{2})\int_0^{T_0}\|q_1(s)-q_2(s)\|_{H^s(\mathbb{R}_+)}\left(\|q_1(s)\|_{H^s(\mathbb{R}_+)}^p+\|q_2(s)\|_{H^s(\mathbb{R}_+)}^p\right)ds\\
\le C(1+T_0^\frac{1}{2})\left(\|q_1(s)\|_{X_{T_0}^s}^p+\|q_2(s)\|_{X_{T_0}^s}^p\right)\int_0^{T_0}\|q_1(s)-q_2(s)\|_{H^s(\mathbb{R}_+)}ds.\end{multline}  Now, unleashing the Gronwall's inequality we get $\|q_1(t)-q_2(t)\|_{H^s(\mathbb{R}_+)}=0$, which implies $q_1\equiv q_2$ on $[0,T_0]$.

\subsubsection*{Continuous dependence}\label{ContDep}
Regarding continuous dependence on data, let $B$ be a bounded subset of $H^s(\mathbb{R}_+)\times H^{\frac{2s+3}{8}}(0,T_0)\times H^{\frac{2s+1}{8}}(0,T_0)$.
Let $(y_0, g_0,g_1)\in B$ and $(z_0,h_0,h_1)\in B$.  Let $y, z$ be two solutions on a common time interval $(0,T_0)$ corresponding to $(y_0, g_0,g_1)$ and$(z_0,h_0,h_1)$, respectively.  Then $w=y-z$ satisfies
\begin{equation}\label{StabilityModel} \left\{ \begin{array}{ll}
         i\partial_t w + \partial_x^4w = F(x,t)\equiv f(y)-f(z), & \mbox{$x\in \mathbb{R}_+$, $t\in (0,T_0)$},\\
         w(x,0)=w_0(x)\equiv (y_0-z_0)(x),\\
         w(0,t)=g(t)\equiv(g_0-h_0)(t),\\
         \partial_xw(0,t)=h(t)\equiv(g_1-h_1)(t).\end{array} \right.
         \end{equation}
Now, using the linear theory together with the nonlinear $H^s$ estimates on the differences, we have
$$\|w\|_{X_{T_0}}\le C\left(\|w_0\|_{H^s(\mathbb{R}_+)}+\|g\|_{H^{\frac{2s+3}{8}}(0,T_0)}+(1+T_0^{\frac{1}{2}})\|h\|_{H^{\frac{2s+1}{8}}(0,T_0)}+\|F\|_{L^1(0,T_0;H^s(\mathbb{R}_+))}\right),$$
where $$\|F\|_{L^1(0,T_0;H^s(\mathbb{R}_+))}\le CT_0\left(\|y\|_{X_{T_0}}^p+\|z\|_{X_{T_0}}^p\right)\|y-z\|_{X_{T_0}}.$$

Choosing $R$, as in the proof of the local existence, and $T_0$ accordingly small enough, we obtain
\begin{equation}\label{ContDep01}\|y-z\|_{X_{T_0}}\le C\left(\|y_0-z_0\|_{H^s(\mathbb{R}_+)}+\|g_0-h_0\|_{H^{\frac{2s+3}{8}}(0,T_0)}+\|g_1-h_1\|_{H^{\frac{2s+1}{8}}(0,T_0)}\right).\end{equation}

\subsubsection*{Blow-up alternative}\label{BlowSec} In this section, we want to obtain a condition which guarantees that a given local solution on $[0,T_0]$ can be extended globally.  Let's consider the set $S$ of all $T_0\in (0,T]$ such that there exists a unique local solution in $X_{T_0}^s$.  We claim that if $\displaystyle T_{max}:=\sup_{T_0\in S}T_0<T$, then $\displaystyle\lim_{t\uparrow T_{max}}\|q(t)\|_{H^s(\mathbb{R}_+)}=\infty.$  In order to prove the claim, assume to the contrary that $\displaystyle\lim_{t\uparrow T_{max}}\|q(t)\|_{H^s(\mathbb{R}_+)}\neq\infty.$  Then $\exists M$ and $t_n\in S$ such that $t_n\rightarrow T_{max}$ and $\|q(t_n)\|_{H^s(\mathbb{R}_+)}\le M.$  For a fixed $n$, we know that there is a unique local solution $q_1$ on $[0,t_n]$.  Now, we consider the following model.

\begin{equation}\label{MainProb3} \left\{ \begin{array}{ll}
         i\partial_t q + \partial_x^4q=f(q), & \mbox{$x\in \mathbb{R}_+$, $t\in (t_n,T)$},\\
         q(x,t_n)=q_1(x,t_n),\\
         q(0,t)=g_0(t),\\
         \partial_xq(0,t)=g_1(t).\end{array} \right.
         \end{equation}
We know from the local existence theory that the above model has a unique local solution $q_2$ on some interval $[t_n,t_n+\delta]$ for some $$\delta=\delta\left(M,\|g_0\|_{H^\frac{2s+3}{8}(0,T)},\|g_1\|_{H^\frac{2s+1}{8}(0,T)}\right)\in (0,T-t_n].$$  Now, choose $n$ sufficiently large that $t_n+\delta>T_{max}$.  If we set \begin{equation} q:=\left\{ \begin{array}{ll}
        q_1, & \mbox{$t\in [0,t_n)$},\\
         q_2, & \mbox{$t\in [t_n,t_n+\delta]$},\end{array} \right.\end{equation} then $q$ is a solution on $[0,t_n+\delta]$ where $t_n+\delta>T_{max}$, which is a contradiction.

Hence, the proof of the local wellposedness is complete.
\subsection{Strichartz estimates and low regularity}
In this section, we consider the nonlinear problem and we assume $0\le s< \frac{1}{2}$ and $(\lambda,r)$ is biharmonic admissible throughout.  We say that the pair $(\lambda,r)$ is \emph{biharmonic admissible} if $$\lambda,r\in [2,\infty] \text{ and }\frac{1}{8}=\frac{1}{4r}+\frac{1}{\lambda}.$$ We also assume that $$p\le \frac{8}{1-2s} \text{ (both \emph{subcritical} '<' and \emph{critical} '=' cases are considered)}.$$
In order to prove wellposedness in this low regularity setting, it is crucial to prove a Strichartz estimate (an estimate of the $L^\lambda_tW^{s,r}_x$ norm) on the solutions of the linear boundary value problem \eqref{3.1}-\eqref{3.4}.  These estimates are generally proven on the whole space or special domains such as manifolds without boundary.  It is in general more difficult to prove these estimates on domains with boundary. However, the representation formula on the half line obtained by the Fokas method provides a suitable kernel structure which will help us to prove the Strichartz estimate.

Note that in this case, we do not need any compatibility condition as the traces of the initial and boundary data are not defined for $s<\frac{1}{2}$.  The solution of \eqref{3.1}-\eqref{3.4} is given by the formula \eqref{zsol2}.  Let us consider the first term in this formula, which is given by \eqref{z1}.  We can rewrite \eqref{z1} as
\begin{multline}\label{z1r1}z_1(x,t)=-\frac{4}{\pi}\int_{0}^\infty e^{-kx+ikx-4ik^4t}\left[(i-1)k^3H_0(4ik^4,T')\right]dk\\
=-\frac{4}{\pi}\int_{0}^\infty e^{-kx+ikx-4ik^4t}\left[(i-1)k^3\int_0^{T'}e^{4ik^4s}h_0(s)ds\right]dk\\
=-\frac{4}{\pi}\int_{0}^\infty e^{-kx+ikx-4ik^4t}\left[(i-1)k^3\int_{-\infty}^{\infty}e^{4ik^4s}h_0(s)ds\right]dk\\
=\frac{4(1-i)}{\pi}\int_{0}^\infty e^{-kx+ikx-4ik^4t}k^3\hat{h}_0(-4k^4)dk\\
=\frac{4(1-i)}{\pi}\int_{-\infty}^\infty \psi(y)\left[\int_{0}^{\infty}e^{-kx+ikx-4ik^4t-iky}dk\right]dy,\end{multline} where $\hat{\psi}(k):=k^3\hat{h}_0(-4k^4)$ if $k\ge 0$ and $\hat{\psi}(k):=0$ otherwise.  We set $$c(x,y,t):=\frac{4(1-i)}{\pi}\int_{0}^{\infty}e^{-kx+ikx-4ik^4t-iky}dk=\frac{4(1-i)}{\pi t^{\frac{1}{4}}}\int_{0}^{\infty}e^{-kxt^{-\frac{1}{4}}+ikxt^{-\frac{1}{4}}-4ik^4-ikyt^{-\frac{1}{4}}}dk.$$ Then $z_1$ is rewritten as $$z_1(x,t)=\int_{-\infty}^\infty c(x,y,t)\psi(y)dy.$$  Let $$\phi(k;x,y,t):=kxt^{-\frac{1}{4}}-4k^4-kyt^{-\frac{1}{4}} \text{ and } p(k;x,t):=e^{-kxt^{-\frac{1}{4}}}.$$  Then, $$c(x,y,t):=\frac{4(1-i)}{\pi t^{\frac{1}{4}}}\int_{0}^{\infty}e^{i\phi(k;x,y,t)}p(k;x,t)dk.$$  Note that $\left|\frac{d^4}{dk^4}\phi(k;x,y,t)\right|=96\ge 1$, which allows us to apply the theory of oscillatory integrals:
\begin{lem}\cite[Section 1.4]{Lin15}
  Let $|\phi^{(4)}(k)|\ge 1$ for $k\ge 0$.  Then $$\left|\int_0^\infty e^{i\phi(k)}p(k)dk\right|\le c\left(\|p\|_{\infty}+\|p\|_1\right).$$
\end{lem}

Applying the above lemma to our kernel, we obtain the pointwise (uniform with respect to $x$ and $y$) estimate $$|c(x,y,t)|\lesssim t^{-\frac{1}{4}}\left(\|p(k;x,t)\|_{L_k^\infty(0,\infty)}+\left\|\frac{d}{dk}p(k;x,t)\right\|_{L_k^1(0,\infty)}\right)\lesssim t^{-\frac{1}{4}}.$$
It follows that \begin{equation}\label{z1int01}\|z_1(\cdot,t)\|_{L^\infty(\mathbb{R}_+)}\lesssim t^{-\frac{1}{4}}\|\psi\|_{L^1(\mathbb{R})}.\end{equation}  Moreover, from \eqref{z1est0}, we already have
\begin{equation}\label{z1int02}\|z_1(\cdot,t)\|_{L^2(\mathbb{R}_+)}\lesssim \|\psi\|_{L^2(\mathbb{R})}.\end{equation}  Interpolating between \eqref{z1int01} and \eqref{z1int02}, we get $$\|z_1(\cdot,t)\|_{L^r(\mathbb{R}_+)}\lesssim t^{-(\frac{1}{4}-\frac{1}{2r})}\|\psi\|_{L^{r'}(\mathbb{R})}.$$

Now, we will estimate $\|z_1\|_{L^\lambda(0,T';L^r(\mathbb{R}_+))}$. To this end, let $\eta\in C_c([0,T'];\mathcal{D}(\mathbb{R}_+))$ be an arbitrary test function.  Then, utilizing the admissibility of $(\lambda,r)$, we have
$$\left|(z_1,\eta)_{L^2(0,T';L^2(\mathbb{R}_+))}\right|\lesssim \|\psi\|_{L^2(\mathbb{R})}\|\eta\|_{L^{\lambda'}(0,T';L^{r'}(\mathbb{R}_+))},$$ from which it follows that \begin{equation}\label{z1int1}\|z_1\|_{L^\lambda(0,T';L^r(\mathbb{R}_+))}\lesssim \|\psi\|_{L^2(\mathbb{R})}\lesssim \|h_0\|_{H^{\frac{3}{8}}(\mathbb{R})}.\end{equation} The first estimate above uses the fact that $$\left\|\int_0^{T'}\int_{0}^\infty \overline{c(x,\cdot,t)}\eta(x,t)dxdt\right\|_{L^2(\mathbb{R})}\lesssim \|\eta\|_{L^{\lambda'}(0,{T'};L^{r'}(\mathbb{R}_+))},$$ which follows from same the arguments in \cite[pp. 25-26]{bona} with the exception that $(\lambda,r)$ satisfies now biharmonic admissibility condition compared to the case of classical Schrödinger equation.  Differentiating in the spatial variable as in \eqref{z1est2} and reapplying the above arguments to the derivative, one obtains \begin{equation}\label{z1int2}\|\partial_x z_1\|_{L^\lambda(0,{T'};L^r(\mathbb{R}_+))}\lesssim \|h_0\|_{H^{\frac{5}{8}}(\mathbb{R})}.\end{equation}  Interpolating between \eqref{z1int1} and \eqref{z1int2}, we get
$$\|z_1\|_{L^\lambda(0,{T'};W^{s,r}(\mathbb{R}_+))}\lesssim \|h_0\|_{H^{\frac{2s+3}{8}}(\mathbb{R})}.$$  Similar estimates can also be found for $z_i$, $i=\overline{2,8}$ in terms of either $\|h_0\|_{H^{\frac{2s+3}{8}}(\mathbb{R})}$ or $\|h_1\|_{H^{\frac{2s+1}{8}}(\mathbb{R})}$, and one obtains the desired Strichartz estimate for the boundary value problem
\begin{equation}\label{zStr1}
  \|z\|_{L^\lambda(0,{T'};W^{s,r}(\mathbb{R}_+))}\lesssim \|h_0\|_{H^{\frac{2s+3}{8}}(\mathbb{R})}+\|h_1\|_{H^{\frac{2s+1}{8}}(\mathbb{R})}.
\end{equation}
Let $y$ be a solution of \eqref{yext} and $w$ be a solution of \eqref{nonhomprob}.  It is well known that the following Strichartz estimates hold true on the whole space \cite{dinh16}:
\begin{equation}\label{yestStr}
\|y\|_{L^\lambda(0,T;{W}^{s,r}(\mathbb{R}))}\lesssim \|q_0^*\|_{{H}^{s}(\mathbb{R})},
\end{equation}
\begin{equation}\label{westStr}
\|w\|_{C([0,T];H^s(\mathbb{R}))}+\|w\|_{L^\lambda(0,T;{W}^{s,r}(\mathbb{R}))}
\lesssim \|f^*\|_{L^{\lambda'}(0,T;{W}^{s,r'}(\mathbb{R}))}.
\end{equation} Moreover, the arguments in the proof of \cite[Proposition 3.8]{audiard} together with the above Strichartz estimates give the following time trace estimates for the solution of the nonhomogeneous Cauchy problem:
\begin{equation}\label{westStr2}
\sup_{x\in\mathbb{R}}\|w(x,\cdot)\|_{H^{\frac{2s+3}{8}}(0,T)} + \sup_{x\in\mathbb{R}}\|\partial_x w(x,\cdot)\|_{H^{\frac{2s+1}{8}}(0,T)}\\
\lesssim \|f^*\|_{L^{\lambda'}(0,T;{W}^{s,r'}(\mathbb{R}))}.
\end{equation}
%
We set \begin{equation}X_{T_0}:=\left\{q\in L^\lambda(0,{T_0};{W}^{s,r}(\mathbb{R}_+))\right.\\
\left.\,|\,\|q\|_{L^\lambda(0,{T_0};{{W}}^{s,r}(\mathbb{R}_+))}\le M\right\}\end{equation}  and equip it with the metric
$$d(q_1,q_2):=\|q_1-q_2\|_{L^\lambda(0,{T_0};{L}^{r}(\mathbb{R}_+))}.$$

\begin{rem}It is well known that $X_{T_0}$ is a complete metric space.  Clearly, $X_{T_0}$ is not a linear space. We will still  write $$|q|_{X_{T_0}}:=\|q\|_{L^\lambda(0,{T_0};{{W}}^{s,r}(\mathbb{R}_+))}$$ to shorten the text in the remaining part of this section.
\end{rem}

Now consider again the operator $\Psi$ given in \eqref{Psi} on the metric space $(X_{T_0},d)$: \begin{equation}\label{Psiywz}
  \Psi(q)=y|_{\mathbb{R}_+}+w|_{\mathbb{R}_+}+z.
\end{equation}
We see from \eqref{yestStr} that \begin{equation}\label{ypsiest}
                                                        |y|_{\mathbb{R}_+}|_{X_{T_0}}\lesssim \|q_0\|_{H^s(\mathbb{R}_+)}.
                                                      \end{equation} Replacing $f^*$ by the nonlinear term $f(q^*)$ in \eqref{westStr}, with $(\cdot)^*$ denoting a fixed bounded extension operator from $H^s(\mathbb{R}_+)\cap {W}^{s,r}(\mathbb{R}_+))$ to $H^s \cap {W}^{s,r}$,  we get
\begin{equation}\label{wpsiest}|w|_{\mathbb{R}_+}|_{X_{T_0}}\lesssim \|f(q^*)\|_{L^{\lambda'}(0,T;{W}^{s,r'}(\mathbb{R}))}.\end{equation}
Moreover, we see from \eqref{zStr1}, \eqref{timeest001}, \eqref{timeest002}, and \eqref{westStr2} that
\begin{multline}\label{zXt0}
|z|_{X_{T_0}}\lesssim \|h_0\|_{H^{\frac{2s+3}{8}}(\mathbb{R})}+\|h_1\|_{H^{\frac{2s+1}{8}}(\mathbb{R})}\\
\lesssim \|g_0\|_{H^{\frac{2s+3}{8}}(0,T_0)}+\|g_1\|_{H^{\frac{2s+1}{8}}(0,T_0)}+c_s(1+T_0^\frac{1}{2})\|q_0\|_{H^s(\mathbb{R}_+)}+\|f(q^*)\|_{L^{\lambda'}(0,T_0;{W}^{s,r'}(\mathbb{R}))}.
\end{multline}
The nonlinear terms at the right hand sides of \eqref{wpsiest} and \eqref{zXt0} can be done by using the particular admissible pair given by $$\lambda:=\frac{8(p+2)}{p(1-2s)},r:=\frac{p+2}{1+sp}.$$ Indeed, one has the following estimates via the fractional Leibniz and chain rules (see the proof of \cite[Theorem 1.1]{Dinh18} for details):
\begin{equation}\label{fqst1}
  \|f(q^*)\|_{{L^{\lambda'}(0,T_0;{\dot{W}}^{s,r'}(\mathbb{R}))}}\lesssim T_0^\theta\|q^*\|_{{L^{\lambda}(0,T_0;{\dot{W}}^{s,r}(\mathbb{R}))}}^{p+1}\le T_0^\theta|q|_{X_{T_0}}^{p+1},
\end{equation}
\begin{equation}\label{fqst2}
  \|f(q^*)\|_{{L^{\lambda'}(0,T_0;{{L}}^{r'}(\mathbb{R}))}}\lesssim T_0^\theta\|q^*\|_{{L^{\lambda}(0,T_0;{\dot{W}}^{s,r}(\mathbb{R}))}}^{p}\|q^*\|_{{L^{\lambda}(0,T_0;{{L}}^{r}(\mathbb{R}))}}\le T_0^{\theta}|q|_{X_{T_0}}^{p+1},
\end{equation}
\begin{multline}\label{fqst3}
  \|f(q_1^*)-f(q_2^*)\|_{{L^{\lambda'}(0,T_0;{{L}}^{r'}(\mathbb{R}))}}\\
  \lesssim T_0^\theta\left(\|q_1^*\|_{{L^{\lambda}(0,T_0;{\dot{W}}^{s,r}(\mathbb{R}))}}^{p}+\|q_2^*\|_{{L^{\lambda}(0,T_0;{\dot{W}}^{s,r}(\mathbb{R}))}}^{p}\right)\|q_1^*-q_2^*\|_{{L^{\lambda}(0,T_0;{{L}}^{r}(\mathbb{R}))}}\\
  \le T_0^\theta\left(|q_1|_{X_{T_0}}^{p}+|q_2|_{X_{T_0}}^{p}\right)d(q_1,q_2),
\end{multline} where $\theta:=1-\frac{p(1-2s)}{8}.$  Hence, assuming $T_0$ is small, say $T_0<1$, there exists some $c>0$ (independent of initial-boundary data and $T_0$) such that
for any $q,q_1,q_2\in X_{T_0}$, one has
\begin{multline}\label{Psi1}
  |\Psi(q)|_{X_{T_0}}\le c(\|g_0\|_{H^{\frac{2s+3}{8}}(0,T_0)}+\|g_1\|_{H^{\frac{2s+1}{8}}(0,T_0)}+\|q_0\|_{H^s(\mathbb{R}_+)})
  +2T_0^\theta|q|_{X_{T_0}}^{p+1}\\
  \le  c(\|g_0\|_{H^{\frac{2s+3}{8}}(0,T_0)}+\|g_1\|_{H^{\frac{2s+1}{8}}(0,T_0)})+\frac{M}{2}+2T_0^\theta M^{p+1},
\end{multline} where $M:=2c\|q_0\|_{H^s(\mathbb{R}_+)}$. Similarly,
\begin{equation}\label{Psi2}
  d(\Psi(q_1),\Psi(q_2))\le cT_0^\theta\left(|q_1|_{X_{T_0}}^{p}+|q_2|_{X_{T_0}}^{p}\right)d(q_1,q_2) \le cT_0^\theta M^{p}d(q_1,q_2).
\end{equation}   In \eqref{Psi1}, the terms $c(\|g_0\|_{H^{\frac{2s+3}{8}}(0,T_0)}+\|g_1\|_{H^{\frac{2s+1}{8}}(0,T_0)})+2T_0^\theta M^{p+1}$ can be made smaller than $\frac{M}{2}$ by choosing $T_0$ small. Similarly, in \eqref{Psi2}, we can guarantee that the product $cT_0^\theta M^{p}<1$ if $T_0$ is chosen small enough.  Now, it is clear that $\Psi$ is a strict contraction on $(X_{T_0},d)$ for small $T_0$ in the subcritical case $\theta>0$ (equivalently if $p<\frac{8}{1-2s}$).

Regarding the critical case, where $\theta = 0$ (equivalently $p=\frac{8}{1-2s}$), we observe that the quantity $|\Psi(q)|_{X_{T_0}}$ is still finite and can be made as small as we wish by choosing $T_0$ small since $|\cdot|_{X_{T_0}}$ is monotone decreasing as $T_0$ decreases. Moreover, we can guarantee that $\Psi$ becomes a contraction on $(X_{T_0},d)$ since the sum $|q_1|_{X_{T_0}}^{p}+|q_2|_{X_{T_0}}^{p}$ can also be made small for small $T_0$.

It follows from \eqref{Psi2} that the uniqueness is guaranteed for small $T_0$, too.

The solution $q\in X_{T_0}$ obtained above in particular belongs to $C([0,T_0];H^s(\mathbb{R}_+))$. Indeed, since $q\in X_{T_0}$ is a fixed point of $\Psi$, we have in view of \eqref{Psiywz}
\begin{equation}\label{qywz}
  q=y|_{\mathbb{R}_+}+w|_{\mathbb{R}_+}+z.
\end{equation} The first term at the right hind side of \eqref{qywz} satisfies
$$\|y|_{\mathbb{R}_+}\|_{C([0,T_0];H^s(\mathbb{R}_+))}\lesssim \|q_0\|_{H^s(\mathbb{R}_+)}$$ as in \eqref{firstest}. The second term satisfies
$$\|w|_{\mathbb{R}_+}\|_{C([0,T_0];H^s(\mathbb{R}_+))}\le \|w\|_{C([0,T_0];H^s(\mathbb{R}))}\lesssim \|f(q^*)\|_{{L^{\lambda'}(0,T_0;{{W}}^{s,r'}(\mathbb{R}))}}\lesssim T_0^{\theta}|q|_{X_{T_0}}^{p+1}.$$ Moreover, from Lemma \ref{lem239} and the second inequality in \eqref{zXt0} , we have
\begin{multline}
  \|z\|_{C([0,T_0];H^s(\mathbb{R}_+))}\\
  \lesssim \|g_0\|_{H^{\frac{2s+3}{8}}(0,T_0)}+\|g_1\|_{H^{\frac{2s+1}{8}}(0,T_0)}+c_s(1+T_0^\frac{1}{2})\|q_0\|_{H^s(\mathbb{R}_+)}+\|f(q^*)\|_{L^{\lambda'}(0,T_0;{W}^{s,r'}(\mathbb{R}))}\\
  \lesssim \|g_0\|_{H^{\frac{2s+3}{8}}(0,T_0)}+\|g_1\|_{H^{\frac{2s+1}{8}}(0,T_0)}+c_s(1+T_0^\frac{1}{2})\|q_0\|_{H^s(\mathbb{R}_+)}+T_0^{\theta}|q|_{X_{T_0}}^{p+1}.
\end{multline}
This completes the proof of Theorem \ref{thmlowreg}.

\subsection{Global wellposedness} In this section, we assume $\kappa\in \mathbb{R}_{-}$.
First we multiply the main equation by $\overline{q}$, use integration by parts on $\mathbb{R}_+$, and take the imaginary parts. Then, we have
\begin{equation}\label{iden00}\frac{1}{2}\frac{d}{dt}\int_0^\infty|q|^2dx=\text{Im}\left[q_{xxx}(0,t)\bar{g}_0(t)\right]-\text{Im}\left[q_{xx}(0,t)\bar{g}_1(t)\right].\end{equation}
As we see from the above identity, the conservation of $L^2$-energy is lost in the presence of inhomogenenous boundary inputs.  Therefore, the global wellposedness is quite a nontrivial problem.  In order to prove the global wellposedness, one needs to gather some information on the second and third order traces.  This enforces us to use other multipliers of higher order.  To this end, we multiply the main equation by $\overline{q}_x$ and use integration by parts on $\mathbb{R}_+$.  Therefore, we first write
\begin{equation}\label{iden01}\int_{0}^\infty q_t\bar{q}_xdx-i\int_{0}^\infty \left(\partial_x^4q\right)\bar{q}_xdx=-i\kappa\int_{0}^\infty|q|^pq\bar{q}_xdx.\end{equation} We can rewrite the first term as
\begin{multline}\int_{0}^\infty q_t\bar{q}_xdx = \frac{d}{dt}\int_{0}^\infty q\bar{q}_{x}dx-\int_{0}^\infty q\bar{q}_{xt}dx\\
= \frac{d}{dt}\int_{0}^\infty q\bar{q}_{x}dx-\left(\int_{0}^\infty \partial_x\left(q\bar{q}_{t}\right)dx-\int_{0}^\infty q_x\bar{q}_{t}dx\right),\end{multline} from which it follows that
\begin{equation}\text{Im}\int_{0}^\infty q_t\bar{q}_xdx =\frac{1}{2}\frac{d}{dt}\text{Im}\int_{0}^\infty {q}\bar{q}_xdx+\frac{1}{2}\text{Im}\left[g_0(t)\bar{g}_0'(t)\right].\end{equation}
Considering the second term in \eqref{iden01}, we have
\begin{multline}\text{Im}\left[i\int_{0}^\infty \left(\partial_x^4q\right)\bar{q}_xdx\right]=\text{Re}\left[\int_{0}^\infty \left(\partial_x^4q\right)\bar{q}_xdx\right]\\
=-\text{Re}\left[q_{xxx}(0,t)\bar{q}_x(0,t)\right]-\frac{1}{2}\int_{0}^\infty \partial_x\left|{q}_{xx}\right|^2dx\\
=-\text{Re}\left[q_{xxx}(0,t)\bar{g}_1(t)\right]+\frac{1}{2}|q_{xx}(0,t)|^2.\end{multline}
Regarding the term at the right hand side of \eqref{iden01}, we have
\begin{equation}\text{Im}\left[i\kappa\int_{0}^\infty|q|^pq\bar{q}_xdx\right]
=\frac{\kappa}{p+2}\text{Re}\left[\int_{0}^\infty\partial_x\left(|q|^{p+2}\right)dx\right]
=-\frac{\kappa}{p+2}|g_0(t)|^{p+2}.\end{equation}
Combining the above, we have
\begin{multline}\label{combinediden01}
  \frac{1}{2}|q_{xx}(0,t)|^2=\frac{1}{2}\frac{d}{dt}\text{Im}\int_{0}^\infty {q}\bar{q}_xdx+\frac{1}{2}\text{Im}\left[g_0(t)\bar{g}_0'(t)\right]\\
  +\text{Re}\left[q_{xxx}(0,t)\bar{g}_1(t)\right]-\frac{\kappa}{p+2}|g_0(t)|^{p+2}.
\end{multline}
Now, we multiply the main equation by $\overline{q}_{t}$ and use integration by parts on $\mathbb{R}_+$.  Therefore, we have
\begin{multline}\label{iden00h2}\frac{d}{dt}\left[\frac{1}{2}\int_0^\infty|q_{xx}|^2dx-\frac{\kappa}{p+2}\int_0^\infty|q|^{p+2}dx\right]\\
=\text{Re}\left[q_{xxx}(0,t)\bar{g}'_0(t)\right]-\text{Re}\left[q_{xx}(0,t)\bar{g}_1'(t)\right].\end{multline}
Next, we multiply the main equation by $\overline{q}_{xxx}$ and use integration by parts on $\mathbb{R}_+$.  Therefore, we first write
\begin{equation}\label{3iden01}\int_{0}^\infty q_t\bar{q}_{xxx}dx-i\int_{0}^\infty \left(\partial_x^4q\right)\bar{q}_{xxx}dx=-i\kappa\int_{0}^\infty|q|^pq\bar{q}_{xxx}dx.\end{equation}
Taking the imarginary part of the first term at the left hand side of \eqref{3iden01} and using integration by parts,  we obtain
\begin{multline}\label{3iden2a}
  \text{Im}\int_{0}^\infty q_t\bar{q}_{xxx}dx = -\text{Im}\left[g_0'(t)\bar{q}_{xx}(0,t)\right ]-\text{Im}\int_{0}^\infty q_{tx}\bar{q}_{xx}dx\\
 = -\text{Im}\left[g_0'(t)\bar{q}_{xx}(0,t)\right ]+\text{Im}\left[g_1'(t)\bar{g}_1(t)\right ]+\text{Im}\int_{0}^\infty q_{txx}\bar{q}_{x}dx\\
 = -\text{Im}\left[g_0'(t)\bar{q}_{xx}(0,t)\right ]+\text{Im}\left[g_1'(t)\bar{g}_1(t)\right ]+\text{Im}\left[q_{xx}(0,t)\bar{g}'_0(t)\right ]\\
 +\text{Im}\int_{0}^\infty q_{xxx}\bar{q}_{t}dx+\frac{d}{dt}\text{Im}\int_{0}^\infty q_{xx}\bar{q}_xdx.
\end{multline}  But recall that $2\text{Im}z=\text{Im}z-\text{Im}\bar{z}$ for any $z\in \mathbb{C}$, and hence we get

\begin{multline}\label{3iden2a}
  \text{Im}\int_{0}^\infty q_t\bar{q}_{xxx}dx =
 = -\frac{1}{2}\text{Im}\left[g_0'(t)\bar{q}_{xx}(0,t)\right ]+\frac{1}{2}\text{Im}\left[g_1'(t)\bar{g}_1(t)\right ]\\+\frac{1}{2}\text{Im}\left[q_{xx}(0,t)\bar{g}'_0(t)\right ]+\frac{1}{2}\frac{d}{dt}\text{Im}\int_{0}^\infty q_{xx}\bar{q}_xdx.
\end{multline}
Considering the second term at the left hand side of \eqref{3iden01}, we have
\begin{multline}\label{3iden02b}
  \text{Im}\left[i\int_{0}^\infty \left(\partial_x^4q\right)\bar{q}_{xxx}dx\right] = \text{Re}\left[\int_{0}^\infty \left(\partial_x^4q\right)\bar{q}_{xxx}dx\right]\\
  =\frac{1}{2}\left[\int_{0}^\infty \partial_x\left(\left|{q}_{xxx}\right|^2\right)dxdt\right]=-\frac{1}{2}\left|{q}_{xxx}(0,t)\right|^2.
\end{multline}
Regarding the term at the right hand side of \eqref{3iden01}, we have
\begin{multline}\text{Im}\left[i\kappa\int_{0}^\infty|q|^pq\bar{q}_{xxx}dx\right]
=\kappa\text{Re}\left[\int_{0}^\infty|q|^pq\bar{q}_{xxx}dx\right]\\
=-\kappa\text{Re}\left[|g_0(t)|^pg_0(t)\bar{q}_{xx}(0,t)\right]\\
-\frac{\kappa(p+2)}{2}\text{Re}\left[\int_{0}^\infty|q|^pq_x\bar{q}_{xx} dx\right]-\frac{\kappa p}{2}\text{Re}\left[\int_{0}^\infty|q|^{p-2}q^2\bar{q}_x\bar{q}_{xx} dx\right].\end{multline}
Combining the three identities above,
\begin{multline}\label{combinediden02}
  \frac{1}{2}|q_{xxx}(0,t)|^2=\frac{1}{2}\text{Im}\left[g_0'(t)\bar{q}_{xx}(0,t)\right ]-\frac{1}{2}\text{Im}\left[g_1'(t)\bar{g}_1(t)\right ]-\frac{1}{2}\text{Im}\left[q_{xx}(0,t)\bar{g}'_0(t)\right ]\\
  -\frac{1}{2}\frac{d}{dt}\text{Im}\int_{0}^\infty q_{xx}\bar{q}_xdx+\kappa\text{Re}\left[|g_0(t)|^pg_0(t)\bar{q}_{xx}(0,t)\right]\\
+\frac{\kappa(p+2)}{2}\text{Re}\left[\int_{0}^\infty|q|^pq_x\bar{q}_{xx} dx\right]+\frac{\kappa p}{2}\text{Re}\left[\int_{0}^\infty|q|^{p-2}q^2\bar{q}_x\bar{q}_{xx} dx\right].
\end{multline}
\subsubsection*{Estimates} Using the $L^2$-identity \eqref{iden00} and the $H^2$-identity \eqref{iden00h2}, we have
\begin{equation}\label{Eest01}
  0\le E(t)\le E(0) + A(t)\|g_0\|_{H^1(0,t)} + B(t)\|g_1\|_{H^1(0,t)}\,,
\end{equation} where
$$E(t):=\frac{1}{2}\int_0^\infty|q|^2dx+\frac{1}{2}\int_0^\infty|q_{xx}|^2dx-\frac{\kappa}{p+2}\int_0^\infty|q|^{p+2}dx\,,$$
$$A(t):=\sqrt{\int_0^t|q_{xxx}(0,\tau)|^2d\tau}\text{ and } B(t):=\sqrt{\int_0^t|q_{xx}(0,\tau)|^2d\tau}.$$
Note that by integration by parts
\begin{equation}\label{H1est01}
  \int_0^\infty|q_x|^2dx = -q(0,t)\bar{q}_x(0,t)-\int_0^\infty q\bar{q}_{xx}dx,
\end{equation}from which it follows that
\begin{multline}\label{H1est02}
  \|q_x\|_{L^2(\mathbb{R}_+)}\le |g_0(t)|^\frac{1}{2}\cdot|g_1(t)|^\frac{1}{2}+\|q\|_{L^2(\mathbb{R}_+)}^\frac{1}{2}\|q_{xx}\|_{L^2(\mathbb{R}_+)}^\frac{1}{2}\\
  \le \|g_0\|_{L^\infty(0,t)}^\frac{1}{2}\|g_1\|_{L^\infty(0,t)}^\frac{1}{2}+\sqrt{E(t)}.
\end{multline}
Using this in \eqref{combinediden01}, we obtain
\begin{multline}\label{combinediden01est}
 B^2(t)\le\|q_0\|_{L^2(\mathbb{R}_+)}\|q_0'\|_{L^2(\mathbb{R}_+)}+\|q\|_{L^2(\mathbb{R}_+)}\|q_x\|_{L^2(\mathbb{R}_+)}+\|g_0\|_{H^1(0,t)}^2\\
  +2A(t)\|g_1\|_{L^2(0,t)}-\frac{2\kappa}{p+2}\|g_0\|_{L^{p+2}(0,t)}^{p+2}\\
  \le \|q_0\|_{H^1(\mathbb{R}_+)}^2+\|g_0\|_{H^1(0,t)}^2-\frac{2\kappa}{p+2}\|g_0\|_{L^{p+2}(0,t)}^{p+2}+\|g_0\|_{L^\infty(0,t)}\|g_1\|_{L^\infty(0,t)}\\
  +{3}{E(t)}+2A(t)\|g_1\|_{L^2(0,t)}\le  c+3E(t)+c A(t),
\end{multline} where $c$ is a nonnegative (generic) constant which might depend on $\kappa,p,$ and various norms of $q_0,g_0,$ and $g_1$.
On the other hand, using \eqref{combinediden02}, we obtain
\begin{multline}\label{combinediden02est}
 A^2(t)\le 2\|g_0'\|_{L^2(0,t)}B(t)+\|g_1\|_{H^1(0,t)}^2+\|q_0''\|_{L^2(\mathbb{R}_+)}\|q_0'\|_{L^2(\mathbb{R}_+)}\\
  +\|q_{xx}\|_{L^2(\mathbb{R}_+)}\|q_x\|_{L^2(\mathbb{R}_+)}+2|\kappa|\cdot\|g_0\|_{L^{2(p+1)}(0,t)}^{2(p+1)}\cdot B(t)\\
+{2(p+1)\int_0^t|\kappa|}\cdot\|q\|_{L^\infty(\mathbb{R}_+)}^p\|q_x\|_{L^2(\mathbb{R}_+)}\|q_{xx}\|_{L^2(\mathbb{R}_+)}ds\\
\le cB(t)+c+2E(t)+c\int_0^t\|q\|_{H^2(\mathbb{R}_+)}^{p+2}ds.
\end{multline}
We deduce from \eqref{H1est01} and the definition of $E(t)$ that $$\|q\|_{H^2(\mathbb{R}_+)}\le c+3\sqrt{2}\sqrt{E(t)}.$$  Using this in \eqref{combinediden02est}, we obtain
\begin{equation}\label{A2}A^2(t)\le cB(t)+c+2E(t)+c\int_0^tE(s)^{\frac{p+2}{2}}ds.\end{equation}
\eqref{Eest01} gives
\begin{equation}\label{son1}
  E^2(t)\le c+cA^2(t)+cB^2(t).
\end{equation}
From \eqref{combinediden01est}, we have
\begin{equation}\label{son2}B^2(t)\le {c}+cE(t)+cA(t).\end{equation}
Combining \eqref{son1} and \eqref{son2}, we obtain
\begin{equation}\label{son3}
  E^2(t)\le c+cA^2(t)+cE(t)+cA(t).
\end{equation}
Applying $\epsilon-$Young's inequality to the above estimate, we get
\begin{equation}\label{son4}
  E^2(t)\le c_\epsilon+\frac{c+\epsilon}{1-\epsilon} A^2(t).
\end{equation}
Using \eqref{son2} in \eqref{A2}, the assumption $p\le 2$, and $\epsilon$-Young's inequality, we get
\begin{equation}\label{son5}A^2(t)\le c_{\epsilon,p,T}+\frac{\epsilon}{1-\epsilon} E^2(t)+\frac{\epsilon c_{p}}{1-\epsilon} \int_0^tE^2(s)ds.\end{equation}
Inserting \eqref{son5} into \eqref{son4}, we can finally close the estimates as
$$\left(1-\frac{\epsilon(c+\epsilon)}{(1-\epsilon)^2}\right)E^2(t)\le c_{\epsilon,p,T}+\frac{\epsilon(c+\epsilon) c_{p}}{(1-\epsilon)^2} \int_0^tE^2(s)ds.$$
Fixing $\epsilon$ to be sufficiently small (so that the coefficients become positive) and unleashing the Gronwall's inequality, we obtain the uniform estimate $E(t)\le c_{\epsilon,p,T}$  for $t\in (0,T]$, which implies
that the solution is global at the $H^2-$level.

\section*{Acknowledgements}
We would like to thank A.S. Fokas (University of Cambridge) for valuable suggestions and also for bringing to our attention reference \cite{dimakos} on the use of the Fokas method in the study of the biharmonic equation in the interior of a convex polygon.  We would like to finally express our gratitude to the anonymous referees whose valuable insights significantly improved the quality of this article.

\bibliographystyle{plain}
\bibliography{mybib}

\end{document}